\documentclass[reqno,12pt]{article}
\usepackage{amssymb,amsthm,amsmath,amsfonts}
\usepackage{epsfig}

\theoremstyle{plain}
\begingroup

\newtheorem{thm}{Theorem}[section]
\newtheorem{theorem}{Theorem}[section]

\newtheorem{corollary}[thm]{Corollary}
\newtheorem{lemma}[thm]{Lemma}

\newtheorem{proposition}[thm]{Proposition}
\newtheorem{conj}[thm]{Conjecture}

\endgroup
\theoremstyle{definition}
\newtheorem{defn}{Definition}[section]
\newtheorem{definition}{Definition}[section]

\newtheorem{remark}[defn]{Remark}

\newtheorem{example}[defn]{Example}

\theoremstyle{remark}

\numberwithin{equation}{section}
\numberwithin{figure}{section}

\DeclareMathOperator{\re}{Re} \DeclareMathOperator{\im}{Im}

\DeclareMathOperator*{\res}{\mathrm{Res}}

\def\I{\mathrm{i}}

\def\D{{\mathbb D}}
\def\R{{\mathbb R}}
\def\C{{\mathbb C}}

\begin{document}

\title{
Non-univalent solutions of the Polubarinova-Galin equation }
\author{
Bj\"orn Gustafsson\textsuperscript{1},
Yu-Lin Lin\textsuperscript{2}\\
}

\date{November 7, 2014}

\maketitle

\begin{abstract}
We study non-univalent solutions of the Polubarinova-Galin equation, describing the time evolution of the conformal map from the unit disk onto
a Hele-Shaw blob of fluid subject to injection at one point. In particular, we tackle the difficulties arising when the map is not even
locally univalent, in which case one has to pass to weak solutions developing on a branched covering surface of the complex plane.

One major concern is the construction of this Riemann surface, which is not given in advance but has to be constantly up-dated along with  
the solution. Once the Riemann surface is constructed the weak solution is automatically global in time, but we have had to leave open the question
whether the weak solution can be kept simply connected all the time (as is necessary to connect to the Polubarinova-Galin equation). 
A certain crucial statement, a kind of stability statement for free boundaries, has therefore been left as a conjecture only.

Another major part of the paper concerns the structure of rational solutions (as for the derivative of the mapping function). Here 
we have fairly complete results on the dynamics. Several examples are given.

\end{abstract}



\footnotetext[1]
{Dept. Mathematics, KTH, 100 44 Stockholm, Sweden. Email: \tt{gbjorn@kth.se}}
\footnotetext[2]
{Dept. Mathematics, KTH, 100 44 Stockholm, Sweden. Email: \tt{ylli@kth.se}}

\tableofcontents


\section{Introduction}

\subsection{General}

This paper is a continuation of \cite{Gustafsson-Lin-2013}, in which the motion of zeros and poles associated to locally univalent solutions of the
Polubarinova-Galin equation was studied. This differential equation, in one real variable (time) and one complex variable, describes the time evolution of a conformal map from the unit disk onto a growing
blob of a viscous fluid squeezed between two parallel plates. The two-dimensional view of the fluid blob is modeled by a domain in the complex plane,
and its growth is assumed to be caused by a source at the origin. The history of this problem goes back to an experiment and a subsequent paper \cite{Hele-Shaw-1898} by Henry Selby Hele-Shaw
in 1898, and the literature on it is by now quite considerable. A short selection is
\cite{Galin-1945},
\cite{Polubarinova-Kochina-1945},
\cite{Vinogradov-Kufarev-1948},
\cite{Saffman-Taylor-1958},
\cite{Richardson-1972}, 
\cite{Shraiman-Bensimon-1984},
\cite{Gustafsson-1984},
\cite{Varchenko-Etingof-1992},
\cite{Reissig-1993}, 
\cite{Hohlov-Howison-Huntingford-Ockendon-Lacey-1994},
\cite{Escher-Simonett-1997},
\cite{Lin-2009a}, \cite{Lin-2009b},
\cite{Hedenmalm-Shimorin-2002},
\cite{Xie-Tanveer-2003},
\cite{Gustafsson-Vasiliev-2006},
\cite{Khavinson-Mineev-Putinar-2009},
\cite{Abanov-Mineev-Zabrodin-2009},
\cite{Ross-Witt-Nystrom-2012},
\cite{Gustafsson-Teodorescu-Vasiliev-2014}.
For the history of the subject we refer to \cite{Vasiliev-2009}.

In the present paper we try to extend previous results on univalent and locally univalent solutions, in particular those in \cite{Gustafsson-Lin-2013},
to the setting of conformal maps which are not even locally univalent. Such mappings  are then considered 
as univalent maps onto subdomains of a suitable Riemann surface, a branched covering surface of the complex plane. 

The Polubarinova-Galin equation for a time dependent normalized conformal map $f(\cdot,t): \D\to \Omega(t)$, where $\Omega(t)$ is the fluid domain at time $t$, reads
\begin{equation}\label{pg0}
{\rm Re}\left[\dot{f}(\zeta,t)\overline{\zeta
f'(\zeta,t)}\right]=q(t) \quad {\rm for}\,\,\zeta\in\partial
\mathbb{D},
\end{equation}
where $q(t)>0$ is the source strength and the normalization means that $f(0,t)=0$, $f'(0,t)>0$. The classical case is that $f$ is univalent, but in this paper
we shall allow arbitrary functions $f$, analytic in some neighborhood of the closed unit disk and subject to the above normalization. 
Then it turns out that it is appropriate to add to (\ref{pg0}) the requirement that
\begin{equation}\label{fomega}
\frac{d}{dt} f(\omega(t),t)=0
\end{equation}
for every zero $\omega(t)$ of $f'(\cdot,t)$ inside $\D$.
With this requirement, (\ref{pg0}) and (\ref{fomega}) taken together become equivalent to an equation of L\"owner-Kufarev type (see (\ref{lk}), (\ref{poisson}) below) and it follows that
the family $f(\cdot,t)$ (for $t$ in some interval) becomes a subordination chain. 
This entails that there exists a Riemann surface $\mathcal{M}$ such that the  $f(\cdot,t)$ become
univalent as mappings into $\mathcal{M}$. As such mappings we put a tilde on the names of functions and domains:
\begin{equation*}\label{omega0}
\tilde{f}(\cdot,t): \D\to\mathcal{M}, \quad \tilde{\Omega}(t)=\tilde{f}(\D,t).
\end{equation*}
Thus we obtain a Hele-Shaw evolution on a Riemann surface.

For (\ref{pg0}), (\ref{fomega}) we start with some given $f(\cdot,0)$, as initial condition. Even if this is taken to be univalent it may happen
that zeros of $f'(\cdot,t)$ reach $\partial\D$ and try to enter $\D$. This causes problems for  (\ref{pg0}) when $q>0$, and one has to pass to a weak
solution in order to allow a zero to make the transition into $\D$. It turns out that the transition is indeed possible, but the solution will then not be smooth
in time. 

A major case under consideration will be when $f'(\zeta,0)$ is a rational function. Then $f'(\zeta,t)$ will remain a rational function for $t>0$, but when a
zero of $f'$ passes through $\partial\D$ it turns out the structure of this rational function changes. In the simplest case it will acquire two new zeros and one pole
of order two. The behavior is quite interesting, and it connects to the theory of contractive divisors on Bergman space \cite{Hedenmalm-Korenblum-Zhu-2000}.

So a considerable part of the paper deals with the structure of rational solutions. Another part of the paper, which however is incomplete at present, is the study of
global in time solutions. The main assertion then is that, given any $f(\cdot,0)$, there exists a weak solution of (\ref{pg0}), (\ref{fomega}) defined for all
$0\leq t<\infty$. The proof of this requires the construction of the appropriate Riemann surface $\mathcal{M}$. This is a not completely trivial task because $\mathcal{M}$ has to be constructed
along with the solution: every time a zero of $f'(\cdot,t)$ reaches $\partial\D$ the Riemann surface has to be updated with a new branch point (as a covering surface of the complex plane).

What we have at present is a complete proof, except for an isolated technical difficulty which still remains. More exactly, what we need to prove is the conjecture stated below.
The prerequisits for the conjecture are as follows.

Let $g$ be a function analytic in a neighborhood of the closed unit disk.
From the theory of quadrature domains \cite{Sakai-1982}, \cite{Sakai-1983}, or the related theory of partial balayage \cite{Gustafsson-Sakai-1994},
\cite{Gardiner-Sjodin-2009}, it follows that for every $t>0$ sufficiently small there exists a domain $D(t)\supset\D$, uniquely determined up to a null-set
and compactly contained in the region of analyticity of $g$, such that (with $dm=dxdy$)
$$
\int_{D(t)}h|g|^2dm \geq \int_{\D}h|g|^2dm+t h(0)
$$
for every function $h$ which is subharmonic and integrable (with respect to $|g|^2m$) in $D(t)$. 
What we will need is that this $D(t)$ is simply
connected if $t>0$ is sufficiently small. In a slightly stronger form this is our conjecture:

\begin{conj}\label{lem:simplyconnected0}If $t>0$ is sufficiently small, then $D(t)$ is star-shaped with respect to the origin, 
in particular simply connected.
\end{conj}

If $g\ne 0$ on $\partial\D$, so that $|g|^2\geq c>0$ in a neighborhood of $\partial\D$, then Conjecture~\ref{lem:simplyconnected0} holds,
as a consequence of results on stability of free boundaries in \cite{Caffarelli-1981}, \cite{Friedman-1982}, for example.
However, we need to use Conjecture~\ref{lem:simplyconnected0} exactly when $g$ has zeros on $\partial\D$. Even in that case
there are strong intuitive and analytic support for Conjecture~\ref{lem:simplyconnected0}, but no rigorous proof that we are aware of. 
The difficulty of proving a statement like Conjecture~\ref{lem:simplyconnected0} was recognized already by Sakai
\cite{Sakai-1988} (Section~5 there).

Despite the fact that we are not been able to settle the above conjecture we believe that the remaining parts of the paper contains enough
interesting material to deserve publication, at least in preprint form. We consider the study of rational solutions (general structure, 
motion of zeros and poles) together with the general set-up for lifting non-univalent solutions to a Riemann surface, which is not
{\it a priori} given, as our our main achievements. In addition, there are a number of enlightening examples.

\subsection{Contents of paper}\label{sec:contents}

A more detailed description of the contents of the paper goes as follows.
In Section~\ref{sec:preparatory} we review the set-up and terminology in the locally univalent case, following essentially \cite{Gustafsson-Lin-2013}.
Section~\ref{sec:nonlocuniv} discusses L\"owner chains and subordination, based on corresponding material in \cite{Pommerenke-1975}, and also
clarifies the relationship between the Polubarinova-Galin equation and the corresponding equation of L\"owner-Kufarev type in the non-locally
univalent case.

Since the global solutions we are looking for  will not be smooth in general  we have to discuss weak solutions, of variational inequality type, which can be formulated in terms
of quadrature domains for subharmonic functions, in the spirit of Sakai \cite{Sakai-1982}, or else in terms of partial balayage. These notions are
explained in some detail in Section~\ref{sec:weaksolutions}, in the planar case, and the corresponding Riemann surface versions are developed in
Section~\ref{sec:lifting}. 

Partial balayage can be considered as an orthogonal projection in a Hilbert space, and when performing this on a covering
surface the question arises to what extent it commutes with the projection map which pushes, for example, measures on the covering space
down to measures on the base space. An affirmative answer to this question is given in Section~\ref{sec:compatibility}. 

The main assertion concerning global in time  weak solutions which stay simply connected all the time
is stated in Section~\ref{sec:universal}, and is proved with
Conjecture~\ref{lem:simplyconnected0} taken as an assumption. In Section~\ref{sec:examples} we elaborate in detail three different solutions 
of the Polubarinova-Galin equation starting out from an cardioid,
and in Section~\ref{sec:rational}, finally, we expose the general structure of rational solutions (the derivative of the mapping function being rational),
and set up the motion of zeros and poles as a dynamical system. Particular emphasis is given to how the structure changes when a zero crosses
the unit circle.

\subsection{Dedication}

We would like to dedicate this paper to the memory of {\it Makoto Sakai}, who passed away in December 2013, at the age of 70.
Sakai made original and groundbreaking contributions in potential theory, highly relevant to the contents
of the present paper. He was one of the creators of the theory of quadrature domains, and his books \cite{Sakai-1982}, \cite{Sakai-2010} 
on the subject  will have a long-standing impact on the development of the subject, and on potential theory and its applications
in general. His papers and books are not always easy to read, but they are very sharp, and Sakai's work is now receiving
increasing recognition in the general mathematical community.

The present paper is closely related to one of Sakai's least known papers, namely \cite{Sakai-1988}. Sakai always
wanted to obtain complete and sharp result, and he was usually successful in this respect. However, in \cite{Sakai-1988} he did not reach that
full perfection, he had to make a probably unnecessary assumption (that a certain domain has no cusps on its boundary)
in order to prove his main result. Ironically, the authors of the present paper were stopped on essentially the same mathematical difficulty,
which we now have left as a conjecture.

\subsection{Acknowledgements}
The authors are grateful to Michiaki Onodera for important information and discussions.


\section{Preparatory material}\label{sec:preparatory}

\subsection{List of notations}\label{sec:notations}

We here list some notations which will be used, but not always further explained, in the
paper.
\begin{itemize}

\item $\mathbb{D}=\{\zeta\in \mathbb{C}:|\zeta|< 1\}$,
$\mathbb{D}(a,r)=\{\zeta\in \mathbb{C}:|\zeta-a|< r\}$.

\item $dm= dm(z)=dx\wedge dy=\frac{1}{2\I}d\bar{z}\wedge dz$
($z=x+\I y$), area measure in the $z$ plane.

\item $\omega^{*}={1}/{\overline{\omega}}$, for $\omega\in \mathbb{C}$.

\item $h^{*}(\zeta)=\overline{h(1/\bar{\zeta})}
=\sum_{j=1}^{m}\overline{b}_{j}\zeta^{-j}$, where
$h(\zeta,t)=\sum_{j=1}^{m}b_{j}\zeta^{j}$.

(There is a slight ambiguity in this notation: we have $\zeta^*=1/\bar\zeta$
if $\zeta$ is considered as a point, whereas $f^*(\zeta)=1/\zeta$ for the
function $f(\zeta)=\zeta$.)

\item $\dot{f}(\zeta,t)=\frac{\partial}{\partial t}f(\zeta,t)$,
$f^{'}(\zeta,t)=\frac{\partial}{\partial\zeta}f(\zeta,t)$.

\item With ${E}\subset\mathbb{C}$ any set which contains the origin,
\begin{align}
\mathcal{O}({E})=&\{f:  \mbox{$f$ is analytic in some neighborhood
of
$E$}\},\notag\\
\mathcal{O}_{\rm norm}({E})=&\{f\in \mathcal{O}({E}):  f(0)=0,
f'(0)>0\},\notag\\
\mathcal{O}_{\rm locu}({E})=&\{f\in \mathcal{O}_{\rm norm}({E}):
f'\neq
0 \,\,\mbox{on}\,\,E\},\notag\\
\mathcal{O}_{\rm univ}({E})=&\{f\in \mathcal{O}_{\rm locu}({E}):
\mbox{$f$ is univalent (one-to-one) on ${E}$}\}.\notag
\end{align}

\item card $=$ `number of elements in'.

\item $SL^1(\Omega, \lambda)$ denotes the set of subharmonic functions in $\Omega$ which are integrable
with respect to a measure $\lambda$.

\item $\nu_f$: counting function, see Definition~\ref{def:countingnumber}.

\item ${\rm Bal\,}(\mu,\lambda)$: partial balayage, see Definition~\ref{def:partialbalayage}.

\item ${\rm supp\,}\nu $: the closed support of a measure, or distribution, $\nu$.

\item $\chi_E$: the characteristic function of a set $E$.

\end{itemize}\par


\subsection{Basic set up in the univalent case}\label{sec:basicequations}

The {\bf Polubarinova-Galin equation} is the dynamical equation for the conformal map
from the unit disk onto a domain in the complex plane representing the two-dimensional view of a blob of  a viscous fluid,
which grows or shrinks due to the presence of a source or sink at one point, chosen to be the origin. The type flow in question, actually incompressible potential flow in the two dimensional
picture, is traditionally called {\bf Hele-Shaw flow} (see \cite{Vasiliev-2009}, \cite{Gustafsson-Teodorescu-Vasiliev-2014}
for historical accounts), and in recent time also {\bf Laplacian growth}, referring to the moving boundary problem.

{\bf }A smooth map $t\mapsto f(\cdot,t)\in \mathcal{O}_{\rm
univ}(\overline{\mathbb{D}})$ is a  \textbf{(strong) solution} of the Polubarinova-Galin
equation if it satisfies
\begin{equation}\label{pg1}
{\rm Re}\left[\dot{f}(\zeta,t)\overline{\zeta
f'(\zeta,t)}\right]=q(t) \quad {\rm for}\,\,\zeta\in\partial
\mathbb{D}.
\end{equation}
Here $q(t)$ is a real-valued function, which is given in
advance and which represents the strength of the source/sink. Typically $q=\pm 1$, which corresponds to injection (plus
sign) or suction (minus sign) at a rate $2\pi$. Since the
transformation $t\mapsto -t$ changes $q$ to $-q$ in (\ref{pg1}) it
is enough to discuss one of the cases $q>0$ and $q<0$. In general we
shall take $q>0$. To increase flexibility we allow $q$ to depend on time, and occasionally also to vanish.

Equation (\ref{pg1}) expresses that the image domains
$\Omega(t)=f(\mathbb{D},t)$ evolve in such a way that
\begin{equation}\label{lg}
\frac{d}{dt}\int_{\Omega(t)} hdm =2\pi{q(t)}h(0)
\end{equation}
for every function $h$ which is harmonic in a neighborhood of
$\overline{\Omega(t)}$. This means that the speed of the boundary
$\partial\Omega(t)$ in the normal direction equals $q(t)$ times the
normal derivative of the Green's function of $\Omega(t)$ with a pole
at $z=0$. The equivalence between (\ref{pg1}) and
(\ref{lg}) follows from the general formula
\begin{equation}\label{generalevolution}
\frac{d}{dt}\int_{\Omega(t)} \varphi \,dm
=\int_{\partial\D}\varphi(f(\zeta,t)){\rm
Re}\left[\dot{f}(\zeta,t)\overline{\zeta f'(\zeta,t)}\right]d\theta
\quad (\zeta=e^{\I\theta}),
\end{equation}
valid for any smooth evolution $t\mapsto f(\cdot,t)\in
\mathcal{O}_{\rm univ}(\overline{\mathbb{D}})$ and for any smooth
test function $\varphi$ in the complex plane. Cf. Lemma~\ref{lemnuf}
below.

On choosing $h(z)=z^k$, $k=0,1,2,\dots$ in (\ref{lg}) it follows
that the harmonic moments
\begin{equation}\label{defmoments}
M_{k}(t)=\frac{1}{\pi}\int_{\Omega(t)}z^{k}dm (z)= \frac{1}{2\pi
\I}\int_{\partial\mathbb{D}}f(\zeta,t)^k f^*(\zeta,t) f'(\zeta,t)d\zeta
\end{equation}
are conserved quantities, except for the first one,
which by (\ref{lg}) is related to $q(t)$ by $\frac{d}{dt}{M}_0(t)=2q(t)$. Thus
\begin{equation}\label{Mzero}
M_0(t)=M_0(0)+2Q(t),
\end{equation}
where $Q(t)$ is the accumulated source up to time $t>0$:
\begin{equation}\label{Q}
Q(t)=\int_0^t q(s)ds.
\end{equation}
Occasionally we may use $Q(t)$ also for $t<0$, in which case it is negative (if $q>0$).

Within the class of smooth (or monotone) evolutions of simply connected domains, Laplacian growth is characterized by the preservation of 
the moments $M_1, M_2, \dots$. This is a consequence of Theorem~10.13 and Corollary~10.14 in \cite{Sakai-1982}. See also Theorem~6.2
and Corollary~6.3 in \cite{Gustafsson-2004}.

One may consider the equation (\ref{pg1}) on different levels of
generality. It is natural to keep the normalization $f(0)=0$,
$f'(0)>0$, in fact the coupling to (\ref{lg}) depends on this, but
(\ref{pg1}) makes sense for any $f\in\mathcal{O}_{\rm
norm}(\overline{\mathbb{D}})$, at least as long one makes sure that
$q(t)=0$ whenever a zero of $f'$ appears on $\partial\mathbb{D}$. 
In the locally univalent case, $f\in\mathcal{O}_{\rm locu}(\overline{\mathbb{D}})$,
the mathematical treatment of (\ref{pg1}) is exactly the
same as in the `physical' case $f\in\mathcal{O}_{\rm univ}(\overline{\mathbb{D}})$. We shall then speak of a
\textbf{locally univalent solution} of the Polubarinova-Galin
equation.

When $f\in\mathcal{O}_{\rm locu}(\overline{\mathbb{D}})$, then
$\dot{f}/\zeta f^{'}\in \mathcal{O}(\overline{\mathbb{D}})$ and
equation (\ref{pg1}) can be solved for $\dot{f}$ by dividing both
sides with $|\zeta f'|^2$. The result is an equation which we shall refer
to as the {\bf L\"owner-Kufarev equation}, namely
\begin{equation}\label{lk}
\dot{f}(\zeta,t)=\zeta f'(\zeta,t)P(\zeta,t)
\quad(\zeta\in\mathbb{D}),
\end{equation}
where $P(\zeta,t)$ is the analytic function in $\D$ whose real part
has boundary value $q(t)|f'(\zeta,t)|^{-2}$ and which is normalized
by $\im P(0,t)=0$. Explicitly $P(\zeta,t)$ is given by
\begin{equation}\label{poisson}
P(\zeta,t) =\frac{1}{2\pi i}
\int_{\partial\D}\frac{q(t)}{|f'(z,t)|^2}\,\frac{z+\zeta}{z-\zeta}\,\frac{dz}{z}
\quad(\zeta\in\mathbb{D}).
\end{equation}
When $f\in\mathcal{O}_{\rm locu}(\overline{\mathbb{D}})$ then
$P\in\mathcal{O}(\overline{\mathbb{D}})$, in fact the right member
of (\ref{lk}) extends analytically as far as $f$ does (see
\cite{Gustafsson-1984}). We shall keep the notation $P=P(\zeta,t)$ also for
the analytic extension of the Poisson integral beyond
$\overline{\mathbb{D}}$.

As a general notation throughout the paper, we set
\begin{equation}\label{derivative}
g(\zeta,t)=f' (\zeta,t).
\end{equation}
The function $g$ in fact turns out to be more fundamental than $f$
itself. Of course, $f$ can be recaptured from $g$ by
$$
f(z,t)=\int_0^z g(\zeta,t)d\zeta.
$$

Part of the paper will deal with the case that $g$ is a rational
function, or perhaps better to say, $g\,d\zeta$ is a rational
differential, in other words an Abelian differential on the Riemann
sphere. If $g$ has residues then $f$ will have logarithmic poles,
besides ordinary poles. The terminology {\bf Abelian domain} for the image
domain $\Omega=f(\mathbb{D})$ has been used \cite{Varchenko-Etingof-1992}
for this case. Alternatively one may speak of $\Omega$ being a
{\bf quadrature domain} (see \cite{Gustafsson-Shapiro-2005} for
the terminology and further references), which in the present case means that a finite quadrature identity of
the kind
\begin{equation}\label{qi}
\int_\Omega h(z)dxdy=\sum_{j=1}^r c_j\int_{\gamma_j}h(z)dz
+\sum_{j=0}^\ell\sum_{k=1}^{n_j-1}a_{jk}h^{(k-1)}(z_j)
\end{equation}
holds for integrable analytic functions $h$ in $\Omega$. Here the
$z_j$ are fixed (i.e., independent of $h$) points in $\Omega$, with specifically $z_0=0$,
the $c_j$, $a_{jk}$ are fixed coefficients, and the
$\gamma_j$ are arcs in $\Omega$ with end points among the $z_j$.
This sort of structure is
stable under Hele-Shaw flow because, as is seen from (\ref{lg}),
what happens under the evolution is only that the right member is
augmented by the term $2\pi Q(t) h(0)$, where $Q(t)$ is the
accumulated source up to time $t$, see (\ref{Q}).

When $g$ is rational we shall write it on the form
\begin{equation}\label{structureg}
g(\zeta,t)
=b(t)\frac{\prod_{k=1}^{m}(\zeta-\omega_{k}(t))}{\prod_{j=1}^{n}(\zeta-\zeta_{j}(t))}
=b(t)\frac{\prod_{i=1}^{m}(\zeta-\omega_{i}(t))}{\prod_{j=1}^{\ell}(\zeta-\zeta_{j}(t))^{n_j}}.
\end{equation}
Here $m\geq n=\sum_{j=1}^\ell {n_j}$, $|\zeta_j|>1$ and repetitions
are allowed among the $\omega_k$, $\zeta_j$ to account for multiple
zeros and poles. Then, with the argument of $b(t)$ chosen so that
$g(0,t)>0$, $f\in\mathcal{O}_{\rm locu}(\overline{\mathbb{D}})$ if
and only if $|\omega_k|>1$, $|\zeta_j|>1$ for all $k$ and $j$. The
assumption $m\geq n$ means that $g\,d\zeta$, as a differential,
has at least a double pole at infinity, which the Hele-Shaw
evolution in any case will force it to have because the source/sink
at the origin creates a pole of $f$ at infinity.

The form (\ref{structureg}) is stable in time, with the sole exception that
when $m=n$ the pole of $f$ may disappear at one moment of time 
(see \cite{Gustafsson-Lin-2013}, or Proposition~\ref{lem:qi} below). 
The rightmost member of (\ref{structureg}) will be used when we need
to be explicit about the orders of the poles. The convention then is
that $\zeta_1,\dots,\zeta_\ell$ are distinct and $n_j\geq 1$. Thus
$n=\sum_{j=1}^\ell {n_j}$, and in the full sequence
$\zeta_1,\dots,\zeta_n$, the tail $\zeta_{\ell+1},\dots,\zeta_n$
will be repetitions of (some of) the $\zeta_1,\dots,\zeta_\ell$
according to their orders. In equations (\ref{qi}) and
(\ref{structureg}), $\ell$ and the $n_j$ are the same.

One can easily express the L\"owner-Kufarev equation (\ref{lk})
directly in terms of $g$, in fact, writing
$P_g$ for the Poisson integral in (\ref{poisson}), (\ref{lk}) is equivalent to
\begin{equation}\label{PGg}
\frac{\partial}{\partial t}\log g(\zeta,t)
=\zeta P_g (\zeta,t)\frac{\partial}{\partial \zeta}\log g (\zeta,t)+\frac{\partial}{\partial \zeta}(\zeta P_g(\zeta,t)).
\end{equation}
When $g$ is rational, as in (\ref{structureg}), also $P_g(\zeta,t)$ will be a rational
function
(see more precisely (\ref{structureP}) in Section~\ref{sec:rational} below), and so will
the derivatives of $\log g$: we have
\begin{equation}\label{logg}
\log g(\zeta,t) =\log
b(t)+{\sum_{k=1}^{m}\log({\zeta-\omega_{k}(t)}})
-{\sum_{j=1}^{n}\log({\zeta-\zeta_{j}(t)}}),
\end{equation}
\begin{equation}\label{loggdot}
\frac{\partial}{\partial t}\log g(\zeta,t)=\frac{\dot{b}(t)}{b(t)}
-{\sum_{k=1}^{m}\frac{\dot{\omega}_k(t)}{\zeta-\omega_{k}(t)}}
+{\sum_{j=1}^{n}\frac{\dot{\zeta}_j(t)}{\zeta-\zeta_{j}(t)}},
\end{equation}
\begin{equation}\label{loggprime}
\frac{\partial}{\partial \zeta}\log g(\zeta,t)=
{\sum_{k=1}^{m}\frac{1}{\zeta-\omega_{k}(t)}}
-{\sum_{j=1}^{n}\frac{1}{\zeta-\zeta_{j}(t)}}.
\end{equation}
Thus (\ref{PGg}) becomes an identity between rational functions.


\section{Dynamics and subordination}\label{sec:nonlocuniv}

\subsection{Generalities}

In the non locally univalent case the Polubarinova-Galin and
L\"owner-Kufarev equations are no longer equivalent. The
L\"owner-Kufarev equation is the stronger one, and solutions to it
can still be viewed as univalent mapping functions, but then onto
subdomains of a Riemann surface. The evolution of these subdomains
is monotone, which amounts to saying that the function family is a
subordination chain. Solutions to the more general
Polubarinova-Galin equation are not unique, but still satisfy a
weaker form of monotonicity, namely monotonicity of the counting
function.

\begin{definition}\label{def:countingnumber}
For any $f\in\mathcal{O}(\overline{\mathbb{D}})$, the {\bf counting
function}, or mapping degree, $\nu_{f}$ of $f$ tells how many times
a value $z\in\mathbb{C}$ is attained by $f$ in $\mathbb{D}$. It is an integer valued function
defined almost everywhere in $\mathbb{C}$ (namely outside
$f(\partial\mathbb{D})$) by
\begin{equation}\label{countingnumber}
\nu_{f}(z)={\rm card\,}\{\zeta\in \mathbb{D}:
f(\zeta)=z\}=\frac{1}{2\pi i}\int_{\partial
\mathbb{D}}d\log\left(f(\zeta)-z\right).
\end{equation}
Clearly, $f$ is univalent in $\mathbb{D}$ if and only if
$0\leq\nu_{f}\leq 1$.
\end{definition}

\begin{definition}\label{def:subordination}
Let $f,g\in\mathcal{O}_{\rm norm}({\mathbb{D}})$. We say that $f$ is
{\bf subordinate} to $g$, and write $f\prec g$, if there exists a
univalent function $\varphi:\mathbb{D}\to\mathbb{D}$ such that
$f=g\circ \varphi$. Note that $\varphi$ is automatically normalized,
hence $\varphi\in\mathcal{O}_{\rm univ}({\mathbb{D}})$.

Let $I\subset [0,\infty)$ be any interval. A map $I\ni t\mapsto
f(\cdot,t)\in \mathcal{O}_{\rm norm}({\mathbb{D}})$ is called a {\bf
subordination chain} on $I$ if $f(\cdot, s)\prec f(\cdot,t)$
whenever $s\leq t$.
\end{definition}

The following lemma shows that by increasing the level of
abstraction (lifting the maps to a Riemann surface), subordination
becomes nothing else than ordinary monotonicity. 
The result is not new (see \cite{Pommerenke-1975} for the classical case
of univalent $f$ and $g$), but we give the proof because it will be a model for how our
specific Riemann surfaces needed for the Hele-Shaw problem will be constructed.

\begin{lemma}\label{lem:riemannsurface}
Let $\{f(\cdot,t)\}_{t\in I}\subset \mathcal{O}_{\rm
norm}({\mathbb{D}})$, where $I\subset [0,\infty)$ is any interval.
Then the following are equivalent.

\begin{itemize}

\item[(i)] $\{f(\cdot,t)\}$ is a subordination chain on $I$.

\item[(ii)] There exists a Riemann surface $\mathcal{M}$, a nonconstant
analytic function $p:\mathcal{M}\to\C$ (`covering map') and univalent analytic
functions
$$
\tilde{f}(\cdot,t):\D\to \mathcal{M} \quad (t\in I)
$$
(`liftings' of the $f(\cdot,t)$) such that
\begin{itemize}
\item[(a)] $f(\zeta,t)=p(\tilde{f}(\zeta,t))$;

\item[(b)] $\tilde{f}(\D,s)\subset \tilde{f}(\D,t)$ for $s\leq t$.
\end{itemize}
\end{itemize}

\end{lemma}

\begin{proof}
The proof that $(ii)$ implies $(i)$ is just a straight-forward
verification, with the subordination functions defined by
\begin{equation}\label{defvarphi}
\varphi(\zeta,s,t)=\tilde{f}^{-1}(\tilde{f}(\zeta,s),t)
\end{equation}
for $s\leq t$, and where $\tilde{f}^{-1}(\zeta,t)$ denotes the inverse of
$\tilde{f}(\zeta,t)$ with respect to $\zeta$.

To prove the opposite, assume $(i)$. We have to construct the
Riemann surface $\mathcal{M}$ and the covering map $p$. For each $t\in I$, let
$\D_t$ be a copy of $\D$ and let $\mathcal{M}_t$ be $\D_t$ considered as an
abstract Riemann surface (for which $\D_t$ serves as coordinate space).
We define a covering map
$$
p(\cdot,t): \mathcal{M}_t\to \C
$$
by declaring that it in the coordinate space $\D_t$ shall be
represented by
$$
f(\cdot,t): \D_t\to \C.
$$
When $s\leq t$ we have the embedding $\varphi(\cdot,s,t):
\D_s\to\D_t$ coming from the assumed subordination, which we on the
level of the abstract Riemann surfaces consider as an inclusion map
\begin{equation}\label{inclusionR}
\mathcal{M}_s\subset \mathcal{M}_t.
\end{equation}
Note that these embeddings and inclusions commute with the covering
maps because of the subordination relations
\begin{equation}\label{subordination}
f(\varphi(\zeta,s,t),t)=f(\zeta,s)  \quad (s\leq t).
\end{equation}

In view of the inclusions (\ref{inclusionR}) we may define
$$
\mathcal{M}=\cup_{t\in I} \mathcal{M}_t.
$$
This is a Riemann surface because each point belongs to some $\mathcal{M}_t$,
and there it has a neighborhood (e.g., all of $\mathcal{M}_t$) which can be
identified with an open subset of the complex plane ($\mathcal{M}_t\cong
\D_t\cong\D$), and the coordinates on $\mathcal{M}$ so obtained are related by
invertible analytic functions (the $\varphi(\cdot,s,t)$). The
covering map $p:\mathcal{M}\to\C$ is defined by declaring that on $\mathcal{M}_t$ it
shall agree with $p(\cdot,t)$. Again, this is consistent.

Finally, the map $f(\cdot,t): \D\to \C$ lifts to
$$
\tilde{f}(\cdot,t): \D\to \mathcal{M}_t\subset \mathcal{M}
$$
by declaring that on identifying $\mathcal{M}_t$ with $\D_t$ it shall simply
be the identity map: $\tilde{f}(\zeta,t)=\zeta\in\D_t$ for
$\zeta\in\D$. Also this is consistent.

In the last picture, the evolution maps $\tilde{f}(\zeta,t)$ become
trivial, while the covering maps are nontrivial
($p(\zeta,t)=f(\zeta,t)$):
$$
\D\stackrel{\mathrm{id}}{\longrightarrow} \D_t
\stackrel{f(\cdot,t)}{\longrightarrow} \C.
$$
For visualization it may however be better to have the view
$$
\D\stackrel{\tilde{f}(\cdot,t)}{\longrightarrow}\mathcal{M}_t
\stackrel{\mathrm{proj}}{\longrightarrow} \C
$$
in which the evolution maps $\tilde{f}(\zeta,t)$ really are liftings
of the $f(\zeta,t)$, while the covering maps
$p(\cdot,t)$ are trivial identifications (local identity maps,
except at branch points).

Now, when $\mathcal{M}$ and $p$ have been constructed the rest of the proof
are easy verifications (omitted).
\end{proof}

\begin{example}
The functions
$$
f(\zeta,t)=\frac{\zeta(t^3\zeta-2t^2+1)}{\zeta-t}
$$
can be shown to make up a non-univalent subordination family on the
interval $1< t <\infty$. The derivative $f'(\zeta,t)$ vanishes at
$\zeta=t^{-1}\in\D$. The Riemann surface $\mathcal{M}$ appearing in
Lemma~\ref{lem:riemannsurface} consists, when visualized as a covering
surface over $\C$, of two copies of $\C$ joined by a branch
point at $f(t^{-1},t)=1$. This example will be further discussed in
Example~\ref{ex:sakai}, where also partial proofs of the above statements can be found.
\end{example}

\begin{lemma}
For $f,g\in\mathcal{O}_{\rm norm}({\mathbb{D}})$, $f\prec g$ implies
$\nu_f\leq \nu_g$ (almost everywhere).
\end{lemma}

\begin{proof}
This is immediate from a change of variable in the integral
appearing in $\nu_f$: assuming $f=g\circ \varphi$ we have, using
that $\varphi$ is univalent,
$$
\nu_f(z)=\frac{1}{2\pi i}\int_{\partial
\mathbb{D}}d\log\left(f(\zeta)-z\right) =\frac{1}{2\pi
i}\int_{\partial \mathbb{D}}d\log\left(g(\varphi(\zeta))-z\right)
$$
$$
=\frac{1}{2\pi i}\int_{\varphi(\partial
\mathbb{D})}d\log\left(g(\zeta)-z\right) \leq \frac{1}{2\pi
i}\int_{\partial \mathbb{D}}d\log\left(g(\zeta)-z\right) =\nu_g(z).
$$
\end{proof}


\subsection{The Polubarinova-Galin versus the L\"owner-Kufarev equation}

The relationship between the Polubarinova-Galin and the
L\"owner-Kufarev equations in the non-univalent case is the
following.

\begin{theorem}\label{lkpg}\label{lem:subordination}
Let $I\ni t\mapsto f(\cdot,t)\in \mathcal{O}_{\rm
norm}(\overline{\mathbb{D}})$ be smooth on some time interval $I$
and assume that $f'\ne 0$ on $\partial\mathbb{D}$ on this 
interval. Then for $q(t)\geq 0$ the following are equivalent.
\begin{itemize}

\item[(i)] $f(\zeta,t)$ solves the L\"owner-Kufarev equation
(\ref{lk}).

\item[(ii)] $f(\zeta,t)$ solves the Polubarinova-Galin equation
(\ref{pg1}) and $\dot{f}(\omega,t)=0$ for every root
$\omega\in\mathbb{D}$ of ${f}'(\omega,t)=0$.

\item[(iii)] $f(\zeta,t)$ solves the Polubarinova-Galin equation
(\ref{pg1}) and $\{f(\cdot,t)\}$ is a subordination chain.

\end{itemize}

\end{theorem}

\begin{remark}
As for $(ii)$, note that $\dot{f}(\omega (t),t)=\frac{d}{dt}f(\omega(t),t)$.
\end{remark}

\begin{proof}

The additional condition in $(ii)$ means more precisely (taking
multiplicities into account) that
\begin{equation}\label{invariant}
\frac{\dot{f}(\zeta,t)}{\zeta
{f}^{'}(\zeta,t)}\in\mathcal{O}(\overline{\mathbb{D}})
\end{equation}
After dividing both members in (\ref{pg1}) by $|\zeta f'(\zeta,t)|^2$ and using the
defining properties (\ref{poisson}), (\ref{derivative}) of
$P(\zeta,t)$ this condition is seen to be exactly what is needed to
pass between (\ref{pg1}) and (\ref{lk}). Thus $(i)$ and $(ii)$ are
equivalent.

Assume next that $(iii)$ holds. That  $\{f(\cdot,t)\}$ is a
subordination chain means that for $s\leq t$ there exist univalent
functions $\varphi (\cdot, s,t):\D\to\D$ such that
(\ref{subordination}) holds. By differentiating
(\ref{subordination}) with respect to $t$ it immediately follows
that (\ref{invariant}) holds. Thus $(iii)$ implies $(ii)$.

We finally prove that  $(i)$ implies $(iii)$. This is done exactly
as in the corresponding proof for L\"owner chains of univalent
functions in Chapter~6 of \cite{Pommerenke-1975}. To construct the
subordination functions $\varphi (\cdot,s,t)$ one considers, for
given $s\geq 0$ and $\zeta\in\mathbb{D}$, the initial value problem
\begin{equation}\label{lk1}
\begin{cases}
\frac{dw}{dt}=-wP(w,t), \quad t\geq s,\\
w(s)=\zeta.
\end{cases}
\end{equation}
It has a unique solution $w=w(t)$ defined on the time interval on
which $f$, and hence $P$, is defined. In terms of $w(t)$ we then
define, for $s\leq t$,
$$
\varphi(\zeta,s,t) =w(t).
$$
Since different trajectories for (\ref{lk1}) never intersect
$\varphi(\zeta,s,t)$ is a univalent function of $\zeta$ in the unit
disk, and using the chain rule and (\ref{lk}) one sees that
$\frac{d}{dt}f(\varphi(\zeta,s,t),t)=0 $. Thus
$f(\varphi(\zeta,s,t),t)$ is constantly equal to its initial value
at $t=s$, which is $f(\zeta,s)$. This proves the subordination.

\end{proof}

The Polubarinova-Galin equation itself is equivalent to a more
general version of the L\"owner-Kufarev equation, as follows.

\begin{theorem}\label{thm:R}
Let $I\ni t\mapsto f(\cdot,t)\in \mathcal{O}_{\rm
norm}(\overline{\mathbb{D}})$ be smooth on some time interval $I$
and assume that $f'\ne 0$ on $\partial\mathbb{D}$ on this time
interval. Then $f(\zeta,t)$ solves the Polubarinova-Galin equation (\ref{pg1}) if
and only if
\begin{equation}\label{lkgen}
\dot{f}(\zeta,t)=\zeta
f'(\zeta,t)\left(P(\zeta,t)+R(\zeta,t)\right),
\end{equation}
where $P$ is the Poisson integral (\ref{poisson}) and where
$R(\zeta,t)$ is any function of the form
\begin{equation}\label{definitionR}
R(\zeta,t)=-\I\im\sum_{\omega_j\in\D}\sum_{k=1}^{r_j}\frac{2B_{jk}(t)}{(-\omega_j(t))^k}
+\sum_{\omega_j\in\D}\sum_{k=1}^{r_j}
\left(  \frac{2B_{jk}(t)}{(\zeta -\omega_j(t))^k}-
\frac{2\overline{B_{jk}(t)}\zeta^k}{(1 -\overline{\omega_j(t)}\zeta)^k}
\right).
\end{equation}
Here $\{\omega_j\}$ are the zeros of $f'$ in $\D$ (necessarily
finitely many), $r_j$ is the order of the zero at $\{\omega_j(t)\}$,
and  $B_{jk}(t)$ are arbitrary smooth functions of $t$.
\end{theorem}

\begin{proof}

The proof of Theorem~\ref{thm:R} is immediate since the additional
term $R(\zeta,t)$ satisfies
$$
\re R(\zeta,t)=0, \quad \zeta\in\partial\D,
$$
$$
\im R(0,t)=0
$$
and is allowed to contain exactly those kinds of singularities in
$\D$ which will be killed by the factor $f'$ in front of it in
(\ref{lkgen}).
The first term in (\ref{definitionR}) is just the normalization
assuring that $\im (P(0,t)+R(0,t))=0$, and the other terms exchange
polar parts between $\D$ and $\C\setminus\overline{\D}$ without
changing the real part on $\partial\D$.

\end{proof}

Theorem~\ref{thm:R} holds for general solutions of the Polubarinova-Galin
equation, but will become of particular interest when we discuss rational
solutions in Section~\ref{sec:rational}.


\section{Weak solutions}\label{sec:weaksolutions}

\subsection{Preliminaries and definition}

Some of our main results will be formulated in terms of variational inequality weak solutions, just called weak solutions
for short, which are expressed in terms of time independent test functions which are subharmonic in the domains $\Omega(t)$.
We shall also need general smooth test functions, like $\Phi$ below.
When such test functions are pulled back to the unit disk via the mapping functions $f$ they become time dependent,
and the time and space derivatives will be coupled.
Indeed, if $\Phi(z)$ is any smooth function in $\C$ then, by the chain rule, the
composed function $\Psi(\zeta,t)=\Phi(f(\zeta,t))$, defined for $\zeta\in \overline{\D}$,
satisfies
\begin{equation}\label{chainrule1}
|f'(\zeta,t)|^2\frac{\partial\Psi}{\partial t}=\dot{f}(\zeta,t)\overline{f'(\zeta,t)}\frac{\partial\Psi}{\partial \zeta}
+\overline{\dot{{f}}(\zeta,t)}{f'(\zeta,t)}\frac{\partial\Psi}{\partial \bar\zeta}.
\end{equation}
When working in $\D$ we shall need test functions $\Psi$ which satisfy just (\ref{chainrule1}) in itself, without necessarily
being of the form $\Phi\circ f$ for some $\Phi$.

\begin{lemma}\label{lemnuf}
For any smooth evolution $t\mapsto f\in\mathcal{O}_{\rm
norm}(\overline{\D})$ and any smooth function $\Psi(\zeta,t)$ which
satisfies (\ref{chainrule1}) we have
\begin{equation}\label{ddtpsif}
\frac{d}{dt}\int_{\mathbb{D}} \Psi(\zeta,t)|f'(\zeta,t)|^2
\,dm(\zeta) =\int_{\partial\D}\Psi (\zeta,t)
{\re}\left[\dot{f}(\zeta,t)\overline{\zeta
f'(\zeta,t)}\right]d\theta,
\end{equation}
where $\zeta=e^{\I\theta}$ in the right member.
\end{lemma}

\begin{proof}
Differentiation under the integral sign gives, using (\ref{chainrule1}),
$$
\frac{d}{dt}\int_{\mathbb{D}} \Psi(\zeta,t)|f'(\zeta,t)|^2 \,dm(\zeta)
=\int_{\D}(\dot{f}\bar{f'}\frac{\partial\Psi}{\partial \zeta}
+\dot{\bar{f}}{f'}\frac{\partial\Psi}{\partial \bar\zeta}+\Psi\dot{f'}\bar{f'}
+\Psi f'\dot{\bar{f'}})\,dm
$$
$$
=\frac{1}{2\I}\int_{\D}(\bar{f'}\frac{\partial}{\partial\zeta}(\dot{f}\Psi)
+f' \frac{\partial}{\partial\bar\zeta}(\dot{\bar{f}}\Psi))d\bar\zeta d\zeta
=\frac{1}{2\I}\int_{\partial\D}\Psi(\dot{\bar{f}}f'd\zeta-\dot{f}\bar{f'}d\bar\zeta)
$$
$$
=\int_{\partial\D}\Psi (\zeta,t)
{\re}\left[\dot{f}(\zeta,t)\overline{\zeta
f'(\zeta,t)}\right]d\theta.
$$
\end{proof}

\begin{corollary}\label{cornuf0}
For any smooth evolution $t\mapsto f(\cdot,t)\in \mathcal{O}_{\rm norm}(\overline{\mathbb{D}}))$ and any smooth function $\Phi$ in
$\C$ we have
\begin{equation}
\label{windingnumber}
\frac{d}{dt}\int_{\mathbb{C}}\Phi(z)\nu_{f(\cdot,t)}(z)dm(z)
=\int_{0}^{2\pi}\Phi (f(\zeta,t))
{\re}\left[\dot{f}(\zeta,t)\overline{\zeta
f'(\zeta,t)}\right]d\theta.
\end{equation}
\end{corollary}

\begin{proof}
Pulling the left member back to the unit disk by means of $f$ gives
$$
\frac{d}{dt}\int_{\mathbb{C}}\Phi(z)\nu_{f(\cdot,t)}(z)\,dm(z)
=\frac{d}{dt}\int_{\mathbb{D}} \Phi(f(\zeta,t))|f'(\zeta,t)|^2
\,dm(\zeta).
$$
Since the composed function $\Psi(\zeta,t)=\Phi(f(\zeta,t))$ satisfies
(\ref{chainrule1}) the corollary follows immediately from Lemma~\ref{lemnuf}.
\end{proof}

Note that Corollary~\ref{cornuf0} is strictly weaker than Lemma~\ref{lemnuf}.
If for example $\nu_{f(\cdot,t)}=2$ on some part of $\C$ then Lemma~\ref{lemnuf}
allows the test function $\Psi$ there to take different values on the two sheets of $f(\D)$
lying above this part, which is not possible for the $\Phi$ in Corollary~\ref{cornuf0}.

When $f(\zeta,t)$ solves the Polubarinova-Galin equation (\ref{pg1}) we get
$$
\frac{d}{dt}\int_{\mathbb{C}}\Phi(z)\nu_{f(\cdot,t)}(z)dm(z)
=q(t)\int_{0}^{2\pi}\Phi (f(e^{\I\theta},t))d\theta.
$$
In particular, applying this to arbitrary $\Phi\geq 0$:

\begin{corollary}\label{nufincreasing}
For any solution $t\mapsto f(\cdot,t)\in \mathcal{O}_{\rm
norm}(\mathbb{D})$ of the Polubarinova-Galin equation (\ref{pg1}) with $q(t)\geq 0$,
$\nu_{f(\cdot,t)}$ is an increasing function of $t$.
\end{corollary}

Specializing (\ref{pg1}), on the other hand, to subharmonic and harmonic test functions (which we then denote $h$)
we obtain, in view of the mean-value properties satisfied by such functions:

\begin{corollary}\label{cornuf}
Let $t\mapsto f(\cdot,t)\in \mathcal{O}_{\rm
norm}(\overline{\mathbb{D}})$ solve the Polubarinova-Galin equation
(\ref{pg1}) with $q(t)\geq 0$. Then
$$
\frac{d}{dt}\int_{\mathbb{C}}h \nu_{f(\cdot,t)}dm\geq 2\pi q(t) h(0)
$$
for any $h$ which is subharmonic in a neighborhood of ${\rm supp\,}\nu_f$. If $h$ is harmonic, equality holds.
\end{corollary}

As a particular case we get the relevant version of moment conservation.
Keeping the rightmost member of (\ref{defmoments}) as definition of the harmonic moments
in the non-univalent case, so that 
\begin{equation*}
M_{k}(t)= \frac{1}{2\pi\I}\int_{\mathbb{D}}f(\zeta,t)^k |f'(\zeta,t)|^2 dm(\zeta)
=\frac{1}{\pi}\int_{\mathbb{C}}z^k \nu_{f(\cdot,t)}(z)dm(z),
\end{equation*}
we have
$$
\frac{d}{dt}M_k(t)=0, \quad k=1,2,3,\dots
$$
under the assumptions of Corollary~\ref{cornuf}.

By integrating the inequality in Corollary~\ref{cornuf} with respect to $t$ we next obtain

\begin{corollary}\label{cor:nufqd}
Whenever $s\leq t$ and  $h$ is subharmonic in a neighborhood of
${\rm supp\,}\nu_f(\cdot, t)$ we have, when $f$ solves the
Polubarinova-Galin equation (\ref{pg1}),
\begin{equation}\label{nufqd}
\int_{\mathbb{C}}h \nu_{f(\cdot,t)}dm-\int_{\mathbb{C}}h
\nu_{f(\cdot,s)}dm\geq 2\pi (Q(t)-Q(s)) h(0),
\end{equation}
where $Q$ is the accumulated source (see (\ref{Q})).
\end{corollary}

This corollary connects to a well-established notion of (variational
inequality) weak solution for the Hele-Shaw problem with a source of strength $q(t)\geq 0$
at the origin. We formulate the definition first for domains (or open sets) in $\C$. It will later be extended
to contexts of Riemann surfaces.

\begin{definition}\label{defweaksolution}
With $I\subset\R$ an interval (of any sort), a family of bounded open sets
$\{\Omega(t)\subset\C: t\in I\}$ is a {\bf weak solution} 
if for any $s,t\in I$ with $s\leq t$, $\Omega(s)\subset\Omega(t)$ and
\begin{equation}\label{weaksolution}
\int_{\Omega(t)}h dm-\int_{\Omega(s)}h dm\geq 2\pi (Q(t)-Q(s)) h(0)
\end{equation}
holds for every $h$ which is subharmonic and integrable in $\Omega(t)$.
If the interval $I$ is of the form $[0,T)$ (or $[0,T]$) then it is enough that (\ref{weaksolution})
holds for $s=0$ to have the full strength of (\ref{weaksolution}).
\end{definition}

Thus a solution of the Polubarinova-Galin equation will be a weak solution
as long as it is univalent, i.e., $0\leq \nu_{f(\cdot,t)}\leq 1$. 
When $\nu_{f(\cdot,t)}$ takes values $\geq 2$ it does not fit into the definition of a weak solution,
but the exceeding parts of $\nu_{f(\cdot,t)}$ can still be swept out to
produce a weak solution. This process, of partial balayage, will shortly be discussed in some detail.

Given any initial bounded open set $\Omega(0)$, a weak solution in the sense of Definition~\ref{defweaksolution}
always exists on the interval
$I=[0,\infty)$, and it is unique up to nullsets. If $\Omega(0)$ is
connected and $0\in\Omega(0)$, then also $\Omega(t)$ is connected
for all $t>0$. However, the domains $\Omega(t)$ need not be simply
connected all the time, hence may be out of reach for the
Polubarinova-Galin equation, although they do become simply
connected for large enough $Q(t)$. See \cite{Gustafsson-Vasiliev-2006}, 
\cite{Gustafsson-Teodorescu-Vasiliev-2014} and references therein.


\subsection{Weak solutions on terms of balayage}

The weak solution can be seen as an instance
of a sweeping process called partial balayage,
which under present circumstances results in quadrature domains for
subharmonic functions. We formulate here this sweeping process in $\C$,
but it can easily be adapted to Riemannian manifolds
of any dimension.  Some general references are \cite{Gustafsson-Sakai-1994},
\cite{Gustafsson-2004},  \cite{Sjodin-2007}, \cite{Gardiner-Sjodin-2009}, 
and (in a sligthly different context) \cite{Roos-2014}.

The fixed data is a measure $\lambda$ which (for the purpose of the present article)
has a bounded density with respect to Lebesgue measure, i.e., satisfies $\lambda\leq Cm$
for some constant $C$, and on a sufficiently large set, e.g., outside a compact set, 
moreover is bounded from below:   
\begin{equation}\label{boundslambda}
\lambda\geq cm
\end{equation}
for some $c>0$. 

\begin{definition}\label{def:partialbalayage}
With $\lambda$ as above, let $\mu$ be a Radon measure with compact support in $\C$.
Then {\bf partial balayage} of $\mu$ to $\lambda$ is defined as
$$
\text{Bal\,}(\mu, \lambda)=\mu + \Delta u,
$$
where $u$ is the smallest non-negative locally integrable function satisfying
\begin{equation}\label{muDelta}
\mu+\Delta u \leq \lambda,
\end{equation}
with $\Delta u$ denoting the distributional Laplacian of $u$.
\end{definition}

The assumption (\ref{boundslambda})
guarantees that there exist functions $u\geq 0$ with compact support satisfying (\ref{muDelta}), and then it follows from
general potential theory that a smallest such $u$ exists, and also that it can be taken to be
lower semicontinuous. See the above references. In particular, the result $\text{Bal\,}(\mu, \lambda)$ of partial balayage will be a measure
with compact support.

Assume for simplicity that also $\mu$ is absolutely continuous with respect to Lebesgue
measure, say $d\mu=\rho \,dm$, and that $\lambda=m$. Then $u$ is to be the smallest of all functions
which satisfy
\begin{equation}\label{twoineq}
\begin{cases}
u\geq 0,\\
\Delta u\leq 1-\rho.
\end{cases}
\end{equation}
This statement constitutes an obstacle problem on ordinary form and it has a unique solution.
This solution can also be characterized by the requirement that the two inequalities in
(\ref{twoineq}) shall hold in the complementary sense
\begin{equation}\label{compl}
u(1-\rho -\Delta u)=0.
\end{equation}

In general, $\text{Bal\,}(\mu, \lambda)$ is squeezed between the two natural bounds,
$$
\min \{\mu,\lambda\}\leq \text{Bal\,}(\mu, \lambda)\leq \lambda,
$$
and the more detailed structure is that
\begin{equation}\label{BalmuOmegageneral}
\text{Bal\,}(\mu, \lambda)=\lambda\chi_\Omega +\mu\chi_{\C\setminus\Omega}.
\end{equation}
Here $\Omega$ denotes the largest open set in which equality holds in (\ref{muDelta}),
in other words, $\Omega=\C\setminus{\rm supp\,}(\lambda-{\rm Bal\,}(\mu,\lambda))$. It is called
the {\bf saturated set}, and it contains the noncoincidence set for the obstacle problem:
$$
\{z\in\C: u(z)>0\}\subset\Omega.
$$
The inclusion may be strict, but under mild conditions the difference set is just a Lebesgue null-set.

In view of (\ref{BalmuOmegageneral}), the saturated set $\Omega$ contains all information of the result of partial balayage.
Another characterization of this set, directly in terms of $\mu$ and $\lambda$, is as follows:
\begin{equation}\label{mulambda}
{\mu}< {\lambda} \text{ on\,\,} {\C\setminus\Omega},
\end{equation}
\begin{equation}\label{subharmqi}
\int_{{\Omega}} h \,d{\mu}\leq \int_{{\Omega}} h \,d{\lambda} \text{ for all } h\in SL^1({\Omega},{\lambda}).
\end{equation}
Here (\ref{mulambda}) shall be interpreted as saying that 
$\C\setminus\Omega\subset{\rm supp\,}((\lambda-\mu)_+)$, in other words that whenever $\mu\geq \lambda$
in some open set $U$ it follows that $U\subset\Omega$. 

Recall from Section~\ref{sec:notations} that $SL^1({\Omega},{\lambda})$
denotes the set of subharmonic functions in $\Omega$ which are integrable with respect to $\lambda$.
This class of test functions can, in (\ref{subharmqi}), be replaced by just all logarithmic kernels
$h(z)=\log|z-a|$ for $a\in\C$ together with all $h(z)=-\log|z-b|$ for $b\in\C\setminus\Omega$, see \cite{Sakai-1982}, \cite{Sakai-1983}.
With these test functions, (\ref{subharmqi}) reduces to the statement that $u\geq 0$ in $\C$, $u=0$ on $\C\setminus\Omega$,
where $u$ now denotes the logarithmic potential of $\mu\chi_\Omega-\lambda\chi_\Omega$ (so that $\Delta u = \lambda \chi_\Omega- \mu\chi_\Omega$).
The proof of the equivalence between (\ref{BalmuOmegageneral}) and (\ref{mulambda}), (\ref{subharmqi}) then becomes straight-forward,
on noting in particular that the above $u$ will be identical with the function $u$ appearing in Definition~\ref{def:partialbalayage}.

We shall mostly consider $\text{Bal\,}(\mu, \lambda)$ in cases when there exists
an open set $D\subset\C$ such that $\mu\geq \lambda$ on $D$, $\mu=0$ outside
$D$. In such cases,
\begin{equation}\label{BalmuOmega}
\text{Bal\,}(\mu, \lambda)=\lambda\chi_\Omega.
\end{equation}
When $\lambda=m$,  
(\ref{subharmqi}) then expresses that $\Omega$ is a {\bf quadrature domain for subharmonic functions} for $\mu$.
This means that $\mu=0$ outside $\Omega$ and that
\begin{equation}\label{hmuhm}
 \int_\Omega h d\mu\leq \int_\Omega h dm
\end{equation}
holds for  all $SL^1({\Omega},m)$,
see  \cite{Sakai-1982} for detailed information.

In terms of partial balayage the weak solution $\Omega(t)$ at
time $t$ is obtained by
$$
\text{Bal\,}(2\pi Q(t)\delta_0+\chi_{\Omega(0)}m,m)=\chi_{\Omega(t)} m \quad (t>0).
$$
Generally speaking, partial balayage destroys information: in for example (\ref{BalmuOmega}), $\Omega$ is
uniquely determined by $\mu$, but a huge amount of different measures $\mu$ give the same $\Omega$.
Therefore the balayage point of view, or the formulation with quadrature domains for subharmonic
functions, embodies the fact that not only does Hele-Shaw
flow preserve harmonic moments, so that the mass distributions $2\pi Q(t)\delta_0+\chi_{\Omega(0)}m$
and $\chi_{\Omega(t)} m$ above are gravi-equivalent, but also that there is a time direction
saying that the first mass distribution  contains more information than the second.
This reflects the fact that Laplacian growth is well-posed in one time direction (increasing $t$ when $q>0$)
but ill-posed in the other, and also reminds of the role of entropy in statistical mechanics, which singles
out one time direction.

\begin{example}
To illustrate the use of partial balayage, we note that
the measure $\nu_{f(\cdot,t)} m$ in (\ref{nufqd}) may be
swept to a measure of the form $\chi_{\Omega(t)}m$:
$\text{Bal\,}(\nu_{f(\cdot,t)}m,m)=\chi_{\Omega(t)}m$.
This is the same as saying that $\int h\nu_{f(\cdot,t)} dm\leq \int_{\Omega(t)}h dm$
for $h\in SL^1(\Omega(t),m)$, as in (\ref{hmuhm}) above.
Taking $s=0$ as initial time and assuming for  simplicity that $f(\cdot,0)$ is univalent, so that
$\nu_{f(\cdot,0)}=\chi_{\Omega(0)}$ with $\Omega(0)=f(\D,0)$, the
inequality (\ref{nufqd})  gives
$$
\int_{\Omega(t)}h dm-\int_{\Omega(0)}h dm\geq 2\pi Q(t) h(0)
$$
for functions $h$ subharmonic in $\Omega(t)$. In other words,
$\{\Omega(t):t\geq 0\}$ is the ordinary weak solution, possibly multiply
connected, with initial domain $\Omega(0)$.

The evolution of $\nu_{f(\cdot,t)}$ can therefore be viewed as a
refinement of the ordinary weak solution,  a refinement in the sense that it
contains more information. One can always pass from
$\nu_{f(\cdot,t)}$ to $\chi_\Omega(t)$ by balayage, but there is
in general no way to recover $\nu_{f(\cdot,t)}$ from $\chi_\Omega(t)$.
An even more refined version of the evolution is obtained
by lifting everything to a Riemann surface over $\C$, which we shall now discuss.
\end{example}


\section{Lifting strong and weak solutions to a Riemann surface}\label{sec:lifting}

\subsection{Hele-Shaw flow on manifolds}

Hele-Shaw flow makes sense on Riemannian manifolds (of any
dimension). The only difference compared to the Euclidean case then
is that the measure $dm=dx\wedge dy$ in, for example, (\ref{lg}) and
(\ref{weaksolution}) shall be replaced by the intrinsic volume form
of the manifold. This also indicates how (\ref{weaksolution})
changes under variable transformations ($dm=dx\wedge dy$ shall be
treated as a $2$-form). We shall need to make these things precise
in the case that the Riemannian manifold is a branched covering
Riemann surface over $\C$, with the metric inherited from the
Euclidean metric on $\C$ via the covering map.

Let $\mathcal{M}$ be a Riemann surface and $p:\mathcal{M}\to \C$ a nonconstant analytic
function, thought of as a, possibly branched, covering map. If
$\tilde{z}=\tilde{x}+\I\tilde{y}$ is a local holomorphic coordinate
on $\mathcal{M}$ and $z=x+\I y$ the usual coordinate on $\C$ then the
Riemannian metric on $\mathcal{M}$ is taken to be the Euclidean metric $|dz|^2=dx^2+dy^2$,
which is lifted to $\mathcal{M}$ by
$p$, i.e.,
\begin{equation}\label{metric}
d\tilde{s}^2=|dp|^2= |p'(\tilde z)|^2(|d\tilde{x}|^2+|d\tilde{y}|^2).
\end{equation}
The intrinsic area form on $\mathcal{M}$ is similarly the pull-back of $dm=dx\wedge dy$ to $\mathcal{M}$, namely
\begin{equation}\label{dtildem}
d\tilde{m} = \frac{1}{2\I}d\bar{p}\wedge dp=|p'(\tilde{z})|^2d\tilde{x}\wedge d\tilde{y}.
\end{equation}
In terms of the Hermitian bilinear form  $d\bar p \otimes dp$ one can
write $d\tilde{s}^2=\re d\bar p \otimes dp $, $d\tilde{m}=\im d\bar p \otimes dp $.

Assume now that $0\in p(\mathcal{M})$ and let $\tilde{0}\in \mathcal{M}$ be a point such
that $p(\tilde{0})=0$. Then we may consider Hele-Shaw evolution on
$M$ with injection (or suction) at $\tilde{0}$. In case of a simply
connected evolution $\tilde{\Omega}(t)$, let
$$
\tilde{f}(\cdot,t):\D\to \tilde{\Omega}(t)\subset \mathcal{M}
$$
be conformal maps with $\tilde{f}(0,t)=\tilde{0}$ and $f'(0,t)>0$,
where $f=p\circ\tilde{f}$ is the projection of $\tilde{f}$ to $\C$,
$$
f(\zeta,t)=p(\tilde{f}(\zeta,t)).
$$
The latter relationship gives
\begin{equation}\label{invariance}
\frac{\dot{f}(\zeta,t)}{f'(\zeta,t)}=\frac{\dot{\tilde
f}(\zeta,t)}{\tilde f'(\zeta,t)},
\end{equation}
which expresses invariance of the Poisson integral (\ref{poisson})
under changes of coordinates. In particular it follows that the
evolution of $\tilde{f}$ is described by
\begin{equation}\label{lkR}
\dot{\tilde f}(\zeta,t)=\zeta \tilde f' (\zeta,t)P_g(\zeta,t),
\end{equation}
where $P_g(\zeta,t)=P(\zeta,t)$ is the Poisson integral (\ref{poisson}) defined,
not in terms of $\tilde{f}'$ but in terms of $g=f'$. The relationship
between $f'$ and $\tilde{f}'$ is
\begin{equation}\label{fpf}
f'(\zeta,t)=p'(\tilde{f}(\zeta,t))\tilde{f}'(\zeta,t).
\end{equation}
If $p$ is thought of as just a local identity map (away from branch
points) then $f'$ and $\tilde{f}'$ are the same.

It should be noted that $\tilde{f}(\cdot, t)$ by definition always
is univalent in $\D$, in particular $\tilde{f}'\ne 0$ in $\D$. If
$f'=0$ at some point in $\D$, then it is the factor $p'(\tilde
{f}(\zeta,t))$ in (\ref{fpf}) that vanishes there. When formulated
as a Polubarinova-Galin equation the evolution of $\tilde{f}$ is
given by
\begin{equation}\label{pgR}
\re\,[\dot{\tilde f}(\zeta,t)\overline{\zeta \tilde
f'(\zeta,t)}]=\frac{q(t)}{|p'(\tilde f(\zeta,t))|^2} \quad
(\zeta\in\partial\D).
\end{equation}
This equation is an immediate consequence of (\ref{lkR}), (\ref{fpf}) and (\ref{poisson}).
It is actually the general form of the Polubarinova-Galin equation on a manifold with 
Riemannian metric given as in (\ref{metric}), even when the integral of $p'$ is not interpreted as a 
covering map.

In general, $\dot{\tilde{f}}$ and
$\zeta\tilde{f}'$ should be interpreted as vectors in the tangent
space of $\mathcal{M}$ at $\tilde{z}=\tilde{f}(\zeta,t)$, while $P_g(\zeta,t)$
is a (complex) scalar. This makes (\ref{lkR}) (and (\ref{invariance}))
meaningful. Similarly, (\ref{pgR}) expresses that
$$
<\dot{\tilde{f}},\zeta\tilde{f}'>_\mathcal{M} =q \quad {\rm on\,\,}
\partial\D,
$$
where $<\cdot,\cdot>_\mathcal{M}$ denotes the (real) inner product on the
tangent space of $\mathcal{M}$. Alternatively, expressed in terms of the form
$d\tilde{m}= \frac{1}{2\I}d\bar p\wedge dp$, (\ref{pgR}) says that
$$
d\tilde{m}(\dot{\tilde{f}},\I\zeta\tilde{f}') =q \quad {\rm on\,\,}
\partial\D,
$$
which can be interpreted as a Poisson bracket relation. This has in some mathematical physics literature, see for example
\cite{Abanov-Mineev-Zabrodin-2009}, \cite{Wiegmann-Zabrodin-2000}, \cite{Gustafsson-Teodorescu-Vasiliev-2014},
been formalized under the name {\it string equation}.

When $f(\cdot,t)$ solves the L\"owner-Kufarev equation it is a
subordination chain by Theorem~\ref{lem:subordination} and hence it
can be lifted to a Riemann surface $\mathcal{M}$ by
Lemma~\ref{lem:riemannsurface}. Most of the previous formulas have
simple formulations on $\mathcal{M}$, for example (\ref{lg}) generalizes to
\begin{equation}\label{lgR}
\frac{d}{dt}\int_{\tilde\Omega(t)} h \,d\tilde{m}=2\pi
q(t) h(\tilde 0),
\end{equation}
for $h$ harmonic in a neighborhood of $\tilde{\Omega}(t)$, and
where $\tilde\Omega(t)=\tilde f (\D,t)$, $f=p\circ\tilde f:\D\to\mathcal{M}\to \C$.
For subharmonic $h$ we have inequality $\geq$. Thus on integrating
(\ref{lgR}) with respect to $t$ we arrive at the natural notion of
weak solution on the Riemann surface $\mathcal{M}$.

\begin{definition}\label{defweaksolutionR}
A family of open sets $\{\tilde{\Omega}(t)\subset \mathcal{M}: t\in I\}$ with
compact closure in $\mathcal{M}$ is a {\bf weak solution} on $\mathcal{M}$ if,
for any $s,t\in I$ with $s<t$, $\tilde\Omega(s)\subset \tilde\Omega(t)$ and
\begin{equation}\label{weaksolutionR}
\int_{\tilde\Omega(t)} h d\tilde{m}-\int_{\tilde\Omega(s)} h d\tilde{m}\geq 2\pi
(Q(t)-Q(s)) h(\tilde 0)
\end{equation}
holds for every $h\in SL^1(\tilde{\Omega}(t), \tilde{m})$.
\end{definition}

In terms of partial balayage on $\mathcal{M}$ (which makes good sense) the property of being a weak solution, i.e.,
(\ref{weaksolutionR}) together with $\tilde\Omega(s)\subset \tilde\Omega(t)$, translates into
\begin{equation}\label{defweaksolutionMbalayage}
\text{Bal\,}(2\pi (Q(t)-Q(s))\delta_{\tilde 0}+\chi_{\tilde{\Omega}(s)}\tilde{m},\tilde{m})
=\chi_{\tilde{\Omega}(t)} \tilde{m} \quad (s<t).
\end{equation}


\subsection{Examples}

\begin{example}\label{ex:usageweaksol}
The purpose of this example is to prepare for how the definition of a weak solution is going to be used in the proof
of Theorem~\ref{thm:universal}. For later clarity we spell out everything quite much in detail.

Choose $s=0$ in Definition~\ref{defweaksolutionR}
and assume that the domain $\tilde{\Omega}(0)=\tilde{f}({\D},0)=\tilde{f}({\D})$ at time $s=0$ 
is obtained by uniformization of some
fixed  $f\in\mathcal{O}_{\rm norm}(\overline{\D})$, as in Section~\ref{sec:nonlocuniv}.
Thus $f=p\circ \tilde{f}$, where $p$ is the covering map $p:\tilde{\Omega}(0)\to \C$
which, in terms of the trivial decomposition
$$
\D\stackrel{\mathrm{id}}{\longrightarrow} \D
\stackrel{f}{\longrightarrow} \C,
$$
is obtained by interpreting the second $\D$ as an abstract Riemann surface, identified as $\tilde{\Omega}(0)$. 
The names of the mappings are then shifted to
$$
\D\stackrel{\tilde{f}}{\longrightarrow} \tilde{\Omega}(0)
\stackrel{p}{\longrightarrow} \C.
$$

By assumption, $f$ is analytic in some larger disk, say in $\D(0,\rho)$, $\rho>1$. Thus the two diagrams extend to
\begin{equation}\label{stack3}
\D(0,\rho)\stackrel{\mathrm{id}}{\longrightarrow} \D(0,\rho)
\stackrel{f}{\longrightarrow} \C,
\end{equation}
\begin{equation}\label{stack4}
\D(0,\rho)\stackrel{\tilde{f}}{\longrightarrow} \mathcal{M}\stackrel{p}{\longrightarrow} \C,
\end{equation}
respectively, which defines the Riemann surface $\mathcal{M}$ as being the conformal image of $\D(0,\rho)$ under $\tilde{f}$.
In particular,  $\tilde{\Omega}(0)\subset \mathcal{M}$, and
for small enough $t>0$ the weak solution with initial domain $\tilde{\Omega}(0)$ will stay
compactly in $\mathcal{M}$. The defining property of the
solution domain $\tilde{\Omega}(t)\supset\tilde{\Omega}(0)$ at time $t>0$ is,
when formulated in terms of the abstract Riemann surface notations of (\ref{stack4}),
\begin{equation}\label{DtD}
\int_{\tilde\Omega(t)} \tilde{h} d\tilde{m}-\int_{\tilde\Omega(0)} \tilde{h} d\tilde{m}\geq 2\pi
Q(t) \tilde{h}(\tilde 0).
\end{equation}
This is to hold for all integrable (with respect to $\tilde{m}$) subharmonic functions $\tilde{h}$ in $\tilde{\Omega}(t)$. 
When the same property is formulated by identifying $\mathcal{M}$ with $\D(0,\rho)$ as in  (\ref{stack3}) it becomes
\begin{equation}\label{weaksolutionexample}
\int_{D(t)} h |g|^2 d{m}
-\int_{\D} h |g|^2 d{m}\geq 2\pi Q(t) h(0),
\end{equation}
to hold for all subharmonic $h$ in $D(t)$, integrable with respect to the measure $|g|^2 dm$.
Here $D(t)=\tilde{f}^{-1}(\tilde{\Omega}(t))\subset \D(0,\rho)$, $g=f'$, and $h=\tilde{h}\circ\tilde{f}$,
which is subharmonic if and only if $\tilde{h}$ is. Note that (by definition, (\ref{ex:usageweaksol})) $d\tilde{m}=p^* (dm)$ in the picture (\ref{stack4}),
which becomes $|g|^2 dm$ in the picture (\ref{stack3}).

The domains $\tilde{\Omega}(t)$ and $D(t)$ are not necessarily simply connected when $t>0$, as they are defined only in terms of a weak solution.
Eventually, however, we want to assert that they are simply connected if $t>0$ is small enough (Conjecture~\ref{lem:simplyconnected}).

\end{example}

The only thing which can make a weak solution break down is that it runs out of the manifold, $\mathcal{M}$.
Then the natural thing to do is to try to extend $\mathcal{M}$ to a larger manifold.
Weak solutions are unique (up to null-sets), but they of course
depend on the choice of $\mathcal{M}$ and $p$. If we take $\mathcal{M}$ to be, for example, a disk $\D(0,a)\subset\C$
($a>0$) then, assuming $Q(t)\to\infty$ as $t\to\infty$, any Hele-Shaw evolution will
eventually run out of $\mathcal{M}$. The following example shows that there are always many
different ways of enlarging $\mathcal{M}$, which then give rise to different Hele-Shaw
evolutions.

\begin{example}\label{ex:R}
Choose a point $a>0$ on the positive real axis, to be used as a
stopping point and also as a branch point. Let $\mathcal{M}=\D(0,a)$ be the disk
reaching out to $a$, and consider it as a Riemann surface with trivial
projection map $p(z)=z$ onto $\C$. Then starting from empty space
a Hele-Shaw flow evolution on $\mathcal{M}$ with injection at the origin
gives a family of growing disks, say  $\Omega(t)=\D(0,at)$
on the time interval $0<t<1$, as a weak (and strong) solution for the source strength
$q(t)=a^2t$ (so that $Q(t)=\frac{1}{2\pi}m(\D(0,at))$). At time
$t=1$ it runs out of $\mathcal{M}$, but it can be continued without any
changes on the trivially extended Riemann surface $\mathcal{M}_1=\C$,
for $0<t<\infty$.

However, it can also be continued in many other ways. Let for
example $\C_1$, $\C_2$ be two copies of $\C$ and consider
$$
\mathcal{M}_2=(\C_1\setminus \{a\})\cup (\C_2\setminus \{a\})\cup\{a\}
$$
as a covering surface of $\C$ with a branch point at $z=a$. The
covering map $p:\mathcal{M}_2\to \C$ identifies any point on $\C_j$ ($j=1,2$)
with the corresponding point on $\C$. This is also true at $z=a$,
but a more accurate description there has to be given in terms of a
local coordinate. We may for example choose a local coordinate
$\tilde{z}$ on $\mathcal{M}_2$ so that $\tilde{z}=0$ corresponds to $z=a$ and, more
precisely, so that
$$
p_2(\tilde{z})=\tilde{z}^2 +a.
$$
Thus, with $z=p_2(\tilde{z})$, $\tilde{z}=\sqrt{z-a}$. This coordinate $\tilde{z}$ is actually a global
coordinate on $\mathcal{M}_2$ and it makes $\mathcal{M}_2$ appear as the classical
Riemann surface of the multivalued function $\sqrt{z-a}$.

In terms of the above coordinate, the area form of $\mathcal{M}_2$ is
$$
d\tilde{m}_{2}= \frac{1}{2\I} d\bar{p}_2\wedge dp_2=4
|\tilde{z}|^2d\tilde{x}d\tilde{y}.
$$
The source point is to be one of the two points $\pm\sqrt{-a}$ on $\mathcal{M}_2$
above $0\in\C$, let it be $\tilde{0}=\I\sqrt{a}$, $\sqrt{a}$
denoting the positive root.
Now the definition of a weak solution on $\mathcal{M}_2$ becomes, explicitly,
$$
4\int_{\tilde\Omega(t)}h(\tilde{z}) |\tilde{z}|^2
d\tilde{x}d\tilde{y}-4\int_{\tilde\Omega(s)}h(\tilde z)
|\tilde{z}|^2d\tilde{x}d\tilde{y}\geq 2\pi (t-s) h(\I\sqrt{a}),
$$
to hold for all integrable (with respect to $\tilde{m}_{2}$) subharmonic
functions $h$ in $\tilde\Omega(t)$. Expressed in the coordinate $\tilde{z}$ it
is thus a weighted Hele-Shaw flow, as discussed in for example
\cite{Hedenmalm-Shimorin-2002}. It exists  for all $0<t<\infty$, but it
is certainly different from the solution $\Omega (t)$ on $\mathcal{M}_1=\C$. For
$t>1$, and when viewed on $\mathcal{M}_2$, part of $\tilde\Omega(t)$
continues on the original (`lower') sheet, say $\C_1$, while part
goes to the `upper' sheet $\C_2$. Hedenmalm and Shimorin \cite{Hedenmalm-Shimorin-2002}
use the terminology `wrapped Hele-shaw flow' when the solution goes up on a Riemann covering surface,
at least in the case when there are no branch points.

\end{example}


\begin{example}\label{ex:sakai}

This example can be viewed as a continuation of the previous example, but
from a different point of view. It is based on an example of Sakai \cite{Sakai-1988},
and it appears also, from a different point of view, in \cite{Hedenmalm-1991}. Let
$$
f(\zeta,t)=b(t)\,\frac{\zeta(\zeta-2t^{-1}+t^{-3})}{\zeta-t},
$$
where $1<t<\infty$ and $b(t)\in\R$ are parameters. The derivative is
\begin{equation}\label{exg}
g(\zeta,t)=b(t)\,\frac{(\zeta-t^{-1})(\zeta-2t+t^{-1})}{(\zeta-t)^2}.
\end{equation}
We see that $g$ has two zeros, $\omega_1(t) =t^{-1}\in \D$ and
$\omega_2(t)= 2t-t^{-1}\in \C\setminus\overline{\D}$, and that
$g(\zeta,t)d\zeta$, as a differential, has  double poles at
$\zeta_1(t)=t =\frac{1}{2}(\omega_1(t)+\omega_2(t))$
and at infinity. The data of $g$ are special in two ways: first of
all $\omega_1$ and $\zeta_1$ are reflections of each other with
respect to the unit circle, and secondly $\omega_2$ is chosen so
that $gd\zeta$ has no residues (which is immediately clear since $f$
has no logarithmic poles).

Since $\omega_1(t)\in\D$,  $f(\cdot,t)$ is not locally univalent in $\D$, but it
generates the same moments as a disk: all moments (defined by the rightmost member in (\ref{defmoments}))
vanish, except the first one which is
$$
M_0(t)=b(t)^2\, \frac{2t^2-1}{t^4}.
$$
More generally, the corresponding quadrature identity is
\begin{equation}\label{contractivedivisor1}
\frac{1}{\pi}\int_\D h(\zeta)|f'(\zeta,t)|^2 dm(\zeta) = b(t)^2 \,\frac{2t^2-1}{t^4}h(0),
\end{equation}
holding for $h$ analytic and integrable in $\D$.
This formula also shows that, for the special choice
$b(t)=\frac{t^2}{\sqrt{2t^2-1}}$, $g(\zeta,t)$ is a contractive
(inner) zero divisor in the sense of Hedenmalm \cite{Hedenmalm-1991},
\cite{Hedenmalm-Korenblum-Zhu-2000}.

Despite (\ref{contractivedivisor1}), $f(\zeta,t)$ in general does not solve the
Polubarinova-Galin equation (\ref{pg1}).
Only for one particular choice of $b(t)$ it does.
This choice is determined by the requirement that $f(\omega_1(t),t)$ shall be time independent.
Since
$$
f(\omega_1(t),t)=f(t^{-1},t)=\frac{b(t)}{t^3}
$$
this condition gives
\begin{equation}\label{bat}
b(t)=at^3,
\end{equation}
where $a$ is a constant. A calculation shows that for
this particular choice of $b(t)$, the Polubarinova-Galin equation indeed holds
with $q(t)=a^2t(4t^2-1)$:
$$
{\rm Re}\left[\dot{f}(\zeta,t)\overline{\zeta
f'(\zeta,t)}\right]=a^2 t{(4t^2-1)} \quad {\rm
for}\,\,\zeta\in\partial \mathbb{D}.
$$
Note that $q(t)>0$.
Also the L\"owner-Kufarev equation holds, because $f(\omega_1(t),t)=a$
is fixed (cf. Theorem~\ref{lem:subordination}).

Now we shall see that, taking $a>0$, the above solution, namely
$$
f(\zeta,t)=\frac{a\zeta(t^3\zeta-2t^2+1)}{\zeta-t},
$$
is exactly the projection under $p_2$ of the evolution on $\mathcal{M}_2$ in Example~\ref{ex:R}.
In fact, since $f(\cdot,t)$ maps the zero $\omega_1(t)\in\D$ of $g(\zeta,t)$
onto the fixed point $a$, $f(\cdot,t)$ lifts to a map $\tilde{f}(\cdot,t)$ into the surface $\mathcal{M}_2$.
Inverting $p_2(\tilde{z})=\tilde{z}^2+a$ gives the explicit expression
$$
\tilde{f}(\zeta,t)=\sqrt{f(\zeta,t)-a}=\sqrt{\frac{at(t\zeta -1)^2}{\zeta-t}}.
$$
This function, for any fixed $t>1$, is univalent, $\tilde{f}(\cdot,t): \D\to \tilde{\Omega}(t)\subset \mathcal{M}_2$,
and as a function of $t$ it represents, in the coordinate $\tilde{z}$, the Hele-Shaw evolution on $\mathcal{M}_2$.
Indeed, it satisfies the Polubarinova-Galin equation on $\mathcal{M}_2$:
$$
{\rm Re}\left[\dot{\tilde{f}}(\zeta,t)\overline{\zeta
\tilde{f}'(\zeta,t)}\right]=\frac{q(t)}{4|\tilde{f}(\zeta,t)|^2}\quad
{\rm for}\,\,\zeta\in\partial \mathbb{D}.
$$
This is an instance of (\ref{pgR}), as $p_2'(\tilde{z})=2\tilde{z}$.

As a summary, we write up in coordinates, $z$ and $\tilde{z}$, and $0<t<\infty$,
the complete evolution in Example~\ref{ex:R}, namely the growing disk which at the
point $a$ climbs up to the Riemann surface $\mathcal{M}_2$:

$(i)$ In terms of $z$, solving the ordinary L\"owner-Kufarev equation (\ref{lk}), (\ref{poisson}),
\begin{equation}\label{ftwocases}
f(\zeta,t)=
\begin{cases}
at\zeta \quad & (0<t<1),\\
\frac{a\zeta(t^3\zeta-2t^2+1)}{\zeta-t} \quad &(1<t<\infty).
\end{cases}
\end{equation}
Notice that both of the expressions above are (real) analytic in $t$, even across the
junction value $t=0$. Thus the combined function $f(\zeta,t)$ is  piecewise real analytic with respect to $t$.

$(ii)$ In terms of $\tilde{z}$, for which we have (\ref{lkR}) and (\ref{pgR}) when $t\ne 1$, and for which
the entire solution (across $t=1$) is a weak solution on $\mathcal{M}_2$,
$$
\tilde{f}(\zeta,t)=
\begin{cases}
\sqrt{a(t\zeta-1)} \quad &(0<t<1),\\
\sqrt{\frac{at(t\zeta -t)^2}{\zeta-t}} \quad &(1<t<\infty).
\end{cases}
$$

The source strength is
\begin{equation}\label{sourcestrength}
q(t)=
\begin{cases}
a^2 t \quad \quad &(0<t<1),\\
a^2 t{(4t^2-1)} \quad &(1<t<\infty).
\end{cases}
\end{equation}
Here we can see a discontinuity of $q(t)$ at $t=1$. However this is harmless, and can be avoided by using another
time parametrization. For example one could define $f(\zeta,t)=at^3\zeta$, $\tilde{f}(\zeta,t)=\sqrt{a(t^3\zeta-1)}$
for $0<t<1$, which gives the same family of domains, just traversed with a different speed. This would give
$q(t)=3a^2t^5$ for $0<t<1$, making $q(t)$ continuous across $t=1$.
\end{example}


\subsection{The Riemann surface solution pulled back to the unit disk}

For a function $h(\zeta,t)$ which is holomorphic in $\zeta$, the
requirement (\ref{chainrule1}), with $\Psi=h$, reduces to the simpler statement
\begin{equation}\label{fvarphi0}
{\dot{h}(\zeta,t)}{ f'(\zeta,t)}
={\dot{f}(\zeta,t)}{ h'(\zeta,t)}.
\end{equation}
This can be viewed as the vanishing of a functional determinant and can alternatively be written as
\begin{equation}\label{fvarphi}
\frac{\dot{h}(\zeta,t)}{\zeta h'(\zeta,t)}
=\frac{\dot{f}(\zeta,t)}{\zeta f'(\zeta,t)},
\end{equation}
where (on dividing by $\zeta$) we also have used that $f(0,t)=0$ for
all $t$. When $f$ solves
the L\"owner-Kufarev equation (\ref{lk}) the right member is holomorphic
in $\D$ and equals $P(\zeta,t)$.
Then (\ref{fvarphi}) means that
$h$ solves the same L\"owner-Kufarev equation as $f$. This
can be interpreted as saying that `$h$ flows with $f$', and it also
follows that the $h(\zeta,t)$ are subordinated by the same functions
as $f(\zeta,t)$:
\begin{equation}\label{subordinationh}
h(\varphi(\zeta,s,t),t)=h(\zeta,s) \quad (s\leq t).
\end{equation}
Here $\varphi(\zeta,s,t)$ are the subordination functions in (\ref{subordination}).
Note that (\ref{subordinationh}), or (\ref{fvarphi0}), implies that $h(0,t)=h(0,s)$.

We can now assert

\begin{proposition}\label{thmqd}
Let $t\mapsto f(\cdot,t)\in \mathcal{O}_{\rm norm}(\mathbb{D})$ be a
smooth evolution on some time interval and assume that $f'\ne 0$ on $\partial\D$ on this time interval.
Then $f(\cdot,t)$ solves the Polubarinova-Galin equation (\ref{pg1}) if and only if
\begin{equation}\label{ddthzetat}
\frac{d}{dt}\int_{\mathbb{D}} h(\zeta,t)|f'(\zeta,t)|^2
\,dm(\zeta)=2\pi q(t) h(0,t).
\end{equation}
for every function $h(\cdot,t)\in
\mathcal{O}(\overline{\mathbb{D}})$ which satisfies (\ref{fvarphi0})
(equivalently, (\ref{chainrule1}) or (\ref{subordinationh})), and it solves the L\"owner-Kufarev
equation (\ref{lk}) if and only if moreover (\ref{invariant}) holds
(equivalently, $f(\cdot,t)$ is a subordination chain).
\end{proposition}

\begin{proof}
The proposition follows immediately from Lemma~\ref{lemnuf} and Theorem~\ref{lem:subordination}
since $\re h|_{\partial \D}$ and $\im h|_{\partial \D}$
range over a dense set of functions in (\ref{ddthzetat}).
\end{proof}

Also the Riemann surface weak formulation (\ref{weaksolutionR}) can, in case relevant domains  are simply connected,
 be pulled back to the unit disk in various ways. 
In Example~\ref{ex:usageweaksol} this was done by pulling the initial domain back to
$\D$, which works well for discussing solutions on a short time interval $0\leq t <\varepsilon$.
However, to discuss global solutions it is better to fix a final
time $t=T$ under consideration, and then pull back the domain
$\tilde{\Omega}(T)\subset \mathcal{M}$ at that time to $\D$, assuming that $\tilde{\Omega}(T)$ is simply connected.
Then all previous domains become subdomains of $\D$.

Thus fixing $T$  and identifying $\tilde{\Omega}(T)$ with $\D$  via
$\tilde{f}(\cdot,T)$, equation (\ref{weaksolutionR}) becomes, for $ s<t\leq T$
and on setting $g=f'$ as usual,
\begin{equation}\label{weakunitdisk}
\int_{D(t,T)} h (z)|g(z,T)|^2 \,dm(z)-\int_{D(s,T)} h (z)|g(z,T)|^2 \,dm(z)
\end{equation}
$$
\geq 2\pi (Q(t)-Q(s)) h(0),
$$
to hold for  $h\in SL^1 (D(t,T),m)$. Here the domains 
$D(s,T)=\tilde{f}^{-1} (\tilde{\Omega}(s),T)$, 
$D(t,T)=\tilde{f}^{-1} (\tilde{\Omega}(t),T)$,
satisfying  $D(s,T)\subset D(t,T)\subset \D$, 
need not be simply connected.
Choosing $t=T=0$ with $s<0$ gives
\begin{equation}\label{Ds}
\int_{\D} h (z)|g(z,0)|^2 \,dm(z) - \int_{D(s)} h(z)|g(z,0)|^2 \,dm(z)\geq-2\pi Q(s) h(0),
\end{equation}
where $D(s)=D(s,0)\subset \D$ and $Q(s)<0$. This is a counterpart of
(\ref{weaksolutionexample}) for negative times. It also connects to
the theory of finite contractive zero divisors: starting, as in Example~\ref{ex:R}, a Hele-Shaw
evolution on a Riemann surface $\mathcal{M}$ from empty space, we have $D(s)=\emptyset$  
at the initial time $s<0$, and then (\ref{Ds}) can be identified with the definition of an inner divisor
(namely $g(z,0)$ in the above equation),
as in \cite{Hedenmalm-1991}, \cite{Hedenmalm-Korenblum-Zhu-2000}.

The weak solution can be coupled to the L\"owner-Kufarev equation only if the domains $\tilde{\Omega}(t)$,
or $D(t,T)$, are simply connected. When this is the case we have $D(t,T)=\varphi(\D,t,T)$, where
$\varphi (\zeta,s,t)$ are the subordination functions associated to the conformal maps
$f(\cdot, t):\D\to \tilde{\Omega}(t)$.
In such a case, and returning to (\ref{weakunitdisk}), choosing  $T=t$ there and making
the variable transformation $z=\varphi(\zeta,s,t)$ in the
last integral, one gets
\begin{equation}\label{weakunitdisk2}
\int_{\D} h (z)|g(z,t)|^2 \,dm(z) -\int_{\D} h
(\varphi(\zeta,s,t))|g(\zeta,s)|^2 \,dm(\zeta)
\end{equation}
$$
\geq 2\pi (Q(t)-Q(s)) h(0),
$$
to hold for $h$ subharmonic and integrable in $\D$. For harmonic $h$ we have
equalities in the above inequalities because both of $\pm h$ are then subharmonic.
The relation (\ref{weakunitdisk2}) can also be obtained directly by
integrating (\ref{ddthzetat}) and using (\ref{subordinationh}).
Note that for time dependent test functions, $h(z,t)$, which satisfy (\ref{subordinationh}), 
the relation (\ref{weakunitdisk2}) takes the simpler form
$$
\int_{\D} h (z,t)|g(z,t)|^2 \,dm(z) -\int_{\D} h(z,s)|g(z,s)|^2 \,dm(z)
\geq 2\pi (Q(t)-Q(s)) h(0,t).
$$

We summarize:

\begin{proposition}\label{prop:weaksolutionLK}
A family $\{f(\cdot,t)\in\mathcal{O}_{\rm norm}(\overline{\mathbb{D}}):0\leq t\leq T\}$
represents a weak solution as in Definition \ref{defweaksolutionR}  (with $I=[0,T]$) 
if an only if it is a subordination chain as in Definition~\ref{def:subordination} and
(\ref{weakunitdisk2}) holds for $0\leq s<t\leq T$.
\end{proposition}


\section{Compatibility between balayage and covering maps}\label{sec:compatibility}
The family of branched covering surfaces over $\C$ form a partially ordered set in
a natural way. Within in each of the surfaces one can perform partial balayage,
sweeping to the area form lifted from $\C$.
Thus  we have two kinds of projection maps, reducing refined objects to
cruder objects containing less information:

$(i)$ The first is the balayage operator
taking, for example, an initial domain $\Omega(0)$ to the domain at a later time
$\Omega(t)$ by sweeping out the accumulated source:
$$
\text{Bal\,}(2\pi Q(t)\delta_{\tilde{0}}+\chi_{{\Omega}(s)}\tilde{m},\tilde{m})
=\chi_{{\Omega}(t)} \tilde{m}.
$$
This map is really an orthogonal projection in a Hilbert space
(e.g., the Sobolev space $H_0^{1}(\mathcal{M})=W_0^{1,2}(\mathcal{M})$ if the Dirac measures are suitably smoothed out).
It is a `horizontal' projection, within each covering surface.

$(ii)$ The second is the branched covering map between two Riemann surfaces, by which a measure
on the higher surface can be pushed down to a measure on the lower surface. One may think of this
as a `vertical' projection.

The aim of the present section is to show that these two projections commute in an appropriate sense.
Let $p:\mathcal{M}\to \C$ be a branched covering map, i.e., $p$ is a nonconstant analytic function.
By $p_*$ we denoted the push-forward map, which can be applied to measures on $\mathcal{M}$, to (parametrized) chains for integration
(simply by composition), etc. Similarly, $p^*$ denotes the pull-back map, which can be applied to functions
and differential forms on $\C$. If for example $\tilde{\Omega}$ is a domain in $\mathcal{M}$, thought of as the
oriented $2$-chain parametrized by some $\tilde{f}:\D\to \mathcal{M}$ ($\tilde{\Omega}=\tilde{f}(\D)$),
then $p_*\tilde{\Omega}$ is the $2$-chain parametrized by  $f=p\circ \tilde{f}:\D\to \C$, which can be thought of
as $\Omega=f(\D)$ with appropriate multiplicities. In other words, $p_*$ takes the measure $\chi_{\tilde{\Omega}}\tilde{m}$,
on $\mathcal{M}$, where $d\tilde{m}=p^*dm=d(p\circ x)\wedge d(p\circ y )$, to $\nu_f m$ on $\C$, $\nu_f$ being the counting function,
Definition~\ref{def:countingnumber}.

Note that $p_*$ and $p^*$ are linear maps on suitable vector spaces and that  they, in some formal sense, are each others adjoints.
For example, for measures $\mu$ with compact support on $\mathcal{M}$ and continuous functions $\varphi$ on $\C$ we have
$$
\int_\C\varphi\,d(p_*\mu)= \int_{\mathcal{M}}(\varphi\circ p)\,d\mu =\int_{\mathcal{M}}(p^*\varphi)\,d\mu.
$$
The first identity here can be used as a definition of $p_*$ when it acts on measures, and $p^*(\varphi)$ is simply defined as $\varphi\circ p$. 

In order to be able to use systematic notations we now denote the complex plane by $\mathcal{M}$, and we call the covering surface $\tilde{\mathcal{M}}$.
This makes the proposition below look like a quite general result (which it in fact is, but we shall only prove it under the stated assumptions).

\begin{proposition}\label{prop:compatibility}
With $p:\tilde{\mathcal{M}}\to \mathcal{M}$ a nonconstant proper analytic map between two Riemann surfaces, where $\mathcal{M}=\C$,
let $\tilde{\mu}$ be a measure with compact support in $\tilde{\mathcal{M}}$,
$\lambda$ a measure on $\mathcal{M}$, absolutely continuous with respect to $m$ and satisfying (\ref{boundslambda}), and let
$\tilde{\lambda}$ be a measure on $\tilde{\mathcal{M}}$ satisfying $\tilde{\lambda}\geq p^*\lambda$. 
Then
$$
{\rm Bal\,}(p_*{\rm Bal\,}(\tilde{\mu}, \tilde{\lambda}), \lambda)={\rm Bal\,}(p_*\tilde{\mu},\lambda).
$$
\end{proposition}

\begin{proof}
Since $\mathcal{M}=\C$ and $p$ is proper, $\tilde{\mathcal{M}}$ will be large enough for ${\rm Bal\,}(\tilde{\mu}, \tilde{\lambda})$ to exist 
and have compact support in $\tilde{\mathcal{M}}$. Set then
$$
\tilde{\nu}={\rm Bal\,}(\tilde{\mu}, \tilde{\lambda}),
$$
$$
{\nu}'={\rm Bal\,}(p_*\tilde{\nu}, {\lambda}),
$$
$$
\mu=p_*\tilde{\mu},
$$
$$
{\nu}={\rm Bal\,}({\mu}, {\lambda})
$$
and we shall show that $\nu'=\nu$.

By the general structure of partial balayage (\ref{BalmuOmegageneral}) we have
\begin{equation}\label{balomegatilde}
\tilde{\nu}=\tilde{\lambda}\chi_{\tilde{\Omega}}+
\tilde{\mu}\chi_{\tilde{\mathcal{M}}\setminus\tilde{\Omega}},
\end{equation}
where $\tilde{\Omega}\subset \tilde{\mathcal{M}}$ is the maximal open set in which $\tilde{\nu}=\tilde{\lambda}$.
Recall (\ref{mulambda}), (\ref{subharmqi}) that this $\tilde{\Omega}$ can also be characterized by
\begin{equation}\label{maximal}
\begin{cases}
\tilde{\mu}< \tilde{\lambda} \text{ on } \tilde{\mathcal{M}}\setminus\tilde{\Omega},\\
\int_{\tilde{\Omega}} \psi \,d\tilde{\mu}\leq \int_{\tilde{\Omega}} \psi \,d\tilde{\lambda} \text{ for all } \psi\in SL^1(\tilde{\Omega},\tilde{\lambda}).
\end{cases}
\end{equation}
Here $SL^1(\tilde{\Omega},\tilde{\lambda})$ may be replaced by a smaller test class, as discussed after (\ref{mulambda}), (\ref{subharmqi}),
to avoid some possible integrability problems below.  

Similarly to the above we have
\begin{equation}\label{balG}
{\nu}'=\lambda\chi_{\Omega'}+
(p_*\tilde{\nu})\,\chi_{{\mathcal{M}}\setminus{\Omega'}}
\end{equation}
where $\Omega'\subset {\mathcal{M}}$ is characterized by
\begin{equation}\label{balG1}
\begin{cases}
p_*\tilde{\nu}< {\lambda} \text{ on }\mathcal{M}\setminus\Omega',\\
\int_{\Omega'} \varphi \,d(p_*\tilde{\nu})\leq \int_{\Omega'} \varphi \,d{\lambda} \quad
(\varphi\in SL^1(\Omega',{\lambda})),
\end{cases}
\end{equation}
and
\begin{equation}\label{balomega}
\nu={\lambda}\chi_{{\Omega}}+
\mu\,\chi_{{\mathcal{M}}\setminus{\Omega}},
\end{equation}
with ${\Omega}\subset \mathcal{M}$ characterized by
$$
\begin{cases}
{\mu}< {\lambda} \text{ on }\mathcal{M}\setminus{\Omega},\\
\int_{{\Omega}} \varphi \,d{\mu}\leq \int_{{\Omega}} \varphi \,d{\lambda}
\quad(\varphi\in SL^1({\Omega},{\lambda})).
\end{cases}
$$

Since $p_*$ is a linear operator (\ref{balomegatilde}) gives
\begin{equation}\label{balG2}
p_*\tilde{\nu}=p_*(\tilde{\lambda}\chi_{\tilde{\Omega}})+
p_*(\tilde{\mu}\chi_{\tilde{\mathcal{M}}\setminus\tilde{\Omega}}).
\end{equation}
By the assumption $\tilde{\lambda}\geq p^*(\lambda)$ we have
$p_* (\chi_{\tilde{\Omega}}\tilde{\lambda})\geq \lambda\chi_{p(\tilde{\Omega})}$.
Thus (\ref{balG2}) shows that $p_*\tilde{\nu}\geq \lambda$ in $p(\tilde{\Omega})$.
It follows that $\nu'\geq\lambda$ in $p(\tilde{\Omega})$, hence
\begin{equation}\label{inclusion}
p(\tilde{\Omega})\subset\Omega'.
\end{equation}

By definition of $p_*\tilde{\nu}$, the second part of (\ref{balG1})  spells out to
$$
\int_{p^{-1}({\Omega'})} (\varphi\circ p) \,d{\tilde{\nu}}\leq \int_{{\Omega'}} \varphi \,d{\lambda}
\quad(\varphi\in SL^1({\Omega'},{\lambda})),
$$
which in view of (\ref{balomegatilde}) gives that
$$
\int_{p^{-1}({\Omega'})\cap\tilde{\Omega}} (\varphi\circ p) \,d{\tilde{\lambda}}
+\int_{p^{-1}({\Omega'})\setminus\tilde{\Omega}} (\varphi\circ p) \,d{\tilde{\mu}}
\leq \int_{{\Omega'}} \varphi \,d{\lambda}
\quad (\varphi\in SL^1({\Omega'},{\lambda})).
$$

Next we take $\psi=p^*\varphi=\varphi\circ p$ in (\ref{maximal}).  This gives
$$
\int_{\tilde{\Omega}} (\varphi\circ p) \,d\tilde{\mu}\leq \int_{\tilde{\Omega}} (\varphi\circ p) \,d\tilde{\lambda}
\quad (\varphi\in SL^1({\Omega'},{\lambda})).
$$
Combining with the previous inequality, and using that $ p^{-1}(\Omega')\supset\tilde{\Omega}$ by (\ref{inclusion}),
gives, for $\varphi\in SL^1({\Omega'},{\lambda})$,
$$
\int_{{\Omega'}} \varphi \,d{\mu}=\int_{p^{-1}({\Omega'})} (\varphi\circ p) \,d\tilde{\mu}
=\int_{\tilde{\Omega}} (\varphi\circ p) \,d\tilde{\mu}+\int_{p^{-1}({\Omega'})\setminus\tilde{\Omega}} (\varphi\circ p) \,d\tilde{\mu}
$$
$$
\leq\int_{\tilde{\Omega}} (\varphi\circ p) \,d\tilde{\lambda}+\int_{p^{-1}({\Omega'})\setminus\tilde{\Omega}} (\varphi\circ p) \,d\tilde{\mu}
\leq \int_{{\Omega'}} \varphi \,d{\lambda} \quad (\varphi\in SL^1({\Omega'},{\lambda})).
$$

In summary,
\begin{equation}\label{Omegaprim}
\int_{\Omega'} \varphi \,d\mu\leq \int_{\Omega'} \varphi \,d{\lambda} \quad (\varphi\in SL^1(\Omega',{\lambda})).
\end{equation}
We also have, by (\ref{balomegatilde}), (\ref{balG1}) and, respectively, (\ref{inclusion}),
$$
p_*(\tilde{\mu}\chi_{\tilde{\mathcal{M}}\setminus\tilde{\Omega}})\leq p_ *\tilde{\nu}<\lambda\quad \text{on\,\,}\mathcal{M}\setminus\Omega',
$$
$$
p_*(\tilde{\mu}\chi_{\tilde{\Omega}})=0\quad \text{in\,\,}\mathcal{M}\setminus\Omega'.
$$
Therefore $\mu=p_*\tilde{\mu}<\lambda$ on $\mathcal{M}\setminus\Omega'$. 
In combination with (\ref{Omegaprim}) this gives
$$
{\nu}=\lambda\chi_{\Omega'}+
\mu\chi_{{\mathcal{M}}\setminus{\Omega'}}.
$$
Now (\ref{balG}), (\ref{balG2}), (\ref{inclusion}) finally give
$$
{\nu}'=\lambda\chi_{\Omega'}+(p_*\tilde{\nu})\,\chi_{{\mathcal{M}}\setminus{\Omega'}}
=\lambda\chi_{\Omega'}+(p_*(\tilde{\lambda}\chi_{\tilde\Omega})+p_*(\tilde{\mu}\chi_{\tilde{\mathcal{M}}\setminus\tilde{\Omega}}))
\chi_{{\mathcal{M}}\setminus{\Omega'}}
$$
$$
=\lambda\chi_{\Omega'}+(p_*\tilde{\mu})\,\chi_{{\mathcal{M}}\setminus{\Omega'}}
=\lambda\chi_{\Omega'}+{\mu}\,\chi_{{\mathcal{M}}\setminus{\Omega'}}=\nu,
$$
as desired.

\end{proof}


\section{Global simply connected weak solutions}\label{sec:universal}

As already mentioned, given $p:\mathcal{M}\to \C$ as in
Section~\ref{sec:lifting} and any initial domain
$\tilde\Omega(0)\subset \mathcal{M}$ with $\tilde 0\in\Omega(0)$, a unique
global weak solution $\{\tilde{\Omega}(t):0\leq t<\infty\}$, in the sense of Definition \ref{defweaksolutionR}, 
exists if just $\mathcal{M}$ is large enough. And if $\mathcal{M}$ is not large enough from outset it may always be extended,
in many ways (cf. Example~\ref{ex:R}), to allow for such a global weak solution.
However, even if the initial domain $\tilde\Omega(0)$ is simply connected the weak
solution will in general not remain simply connected all the time.

Now, our main statement, Theorem \ref{thm:universal}, asserts that if $\tilde\Omega(0)$ is simply connected
and has analytic boundary, then it is indeed always possible to choose $\mathcal{M}\supset\tilde\Omega(0)$ so that
the solution $\tilde\Omega(t)$ in $\mathcal{M}$ remains simply connected all the time.
Without referring to any Riemann surface the assertion may be formulated simply as saying that there exists
a global weak solution of the L\"owner-Kufarev equation, for any given $f(\cdot,0)\in\mathcal{O}_{\rm norm}(\overline{\mathbb{D}})$.
The solution cannot not always be smooth in $t$,  because if zeros of $g=f'$ reach the unit circle then it is in most cases
necessary to change the structure of $g$ in order to make the solution go on. The Riemann surfaces involved are needed mainly to make 
the appropriate notion  of a weak solution precise (Definition \ref{defweaksolutionR}). 
  
The difficulty in constructing $\mathcal{M}$ lies in the fact
that it cannot be constructed right away, but has to be created
along with the solution. It has to be updated every time a zero of
$g$ for the corresponding L\"owner-Kufarev equation reaches the
unit circle. 
Unfortunately, as we have not been able to settle Conjecture~\ref{lem:simplyconnected0} 
stated in the introduction, we have to include the validity of this conjecture among the assumptions in the theorem below.
The precise formulation is as follows.

\begin{theorem}\label{thm:universal}
Let $f(\cdot,0)\in\mathcal{O}_{\rm norm}(\overline{\mathbb{D}})$ be
given, together with $q(t)\geq 0$ ($0\leq t<\infty$) such that $Q(t)\to\infty $ as
$t\to\infty$. Then, under the assumption that Conjecture~\ref{lem:simplyconnected0} 
(or Conjecture~\ref{lem:simplyconnected} below) is true, there exists a Riemann surface $\mathcal{M}$, a nonconstant
holomorphic function $p:\mathcal{M}\to \C$ and a point $\tilde{0}\in \mathcal{M}$ with
$p(\tilde{0})=0$ such that the following assertions hold.

\begin{itemize}

\item [(i)] $f(\cdot,0)$ factorizes over $\mathcal{M}$, i.e., there exists a univalent
function $\tilde{f}(\cdot,0):\D\to \mathcal{M}$ with
$\tilde{f}({0},0)=\tilde{0}$ such that
$f(\cdot,0)=p(\tilde{f}(\cdot,0))$.

\item[(ii)] On setting $\tilde{\Omega}(0)=\tilde{f}(\D,0)$, the weak
Hele-Shaw evolution $\{\tilde{\Omega}(t)\}$ on $\mathcal{M}$ with initial
domain $\tilde{\Omega}(0)$ exists for all $0\leq t<\infty$ and
$\tilde{\Omega}(t)$ is simply connected for each $t$. 

\item[(iii)] Let $\nu_{f(\cdot,t)}$ denote the counting function of
$f(\cdot,t)=p(\tilde{f}(\cdot,t))$
and let $\Omega(t)$ denote the domain obtained by partial balayage of
$\nu_{f(\cdot,t)}m$ onto Lebesgue measure $m$:
$$
{\rm Bal\,}(\nu_{f(\cdot,t)}m, m)=\chi_{\Omega(t)} m.
$$
Then the family $\{\Omega(t)\}$ is a weak solution in the ordinary sense
on $\C$, with the domains $\Omega(t)$ possibly multiply connected.

\end{itemize}

\end{theorem}

For the proof of Theorem~\ref{thm:universal} we shall need a few auxiliary results,
stated below.

\begin{lemma}\label{lem:radius}
Let $f\in \mathcal{O}_{\rm norm}(\overline{\mathbb{D}})$ and let
$0<r<1$. Then the following are equivalent.

\begin{itemize}

\item[(i)] $f$ extends to be meromorphic in ${\D(0,\frac{1}{r})}$ with
poles only at the reflected  (in $\partial\D$) zeros of $g$, more
precisely so that $fg^*\in \mathcal{O}(\D(0,\frac{1}{r})\setminus
\overline{\D})$.

\item[(ii)] For every number  $\rho$ with $r<\rho<1$ there exists a
constant $C_\rho$ such that
\begin{equation}\label{crho}
|\int_\D h|g|^2 dm| \leq C_\rho \sup_{\D(0,\rho)} |h| \quad
(h\in\mathcal{O}(\overline{\mathbb{D}})).
\end{equation}

\item[(iii)] For every number  $\rho$ with $r<\rho<1$ there exists a
(signed) measure $\sigma$ with ${\rm supp\,}\sigma\subset
\overline{\D(0,\rho)}$ such that
\begin{equation}\label{sigma}
\int_\D h|g|^2 dm = \int hd\sigma \quad
(h\in\mathcal{O}(\overline{\mathbb{D}})).
\end{equation}

\end{itemize}

\end{lemma}

\begin{proof}
Assume $(i)$. Then for every $r<\rho<1$ we have
$$
\int_\D h|g|^2 dm=\frac{1}{2\I}\int_{\partial\D} h\bar f df
=\frac{1}{2\I}\int_{\partial\D} h f^*g
d\zeta=\frac{1}{2\I}\int_{\partial\D(0,\rho)}h f^*g d\zeta,
$$
where we used that $f^*g=(fg^*)^*\in \mathcal{O}(\D\setminus\overline{\D(0,{r})})$,
by assumption. Now $(ii)$ follows with $C_\rho=\frac{1}{2}\int_{\partial\D(0,\rho)}| f^*g|| d\zeta|$.

That $(ii)$ implies $(iii)$ follows from general functional analysis
(the Hahn-Banach theorem and the Riesz representation theorem for
functionals on $C(\overline{\D(0,\rho)})$, see \cite{Rudin-1987}).

Assume now $(iii)$ and we shall prove $(i)$. Consider the Cauchy transforms of $\sigma$ and,
$|g|^2\chi_\D$, defined by
$$
\hat{\sigma}(z)=\frac{1}{\pi}\int_\D \frac{
d\sigma(\zeta)}{z-\zeta},
$$
\begin{equation}\label{cauchyG}
G(z)=\frac{1}{\pi}\int_\D \frac{ |g(\zeta)|^2 dm(\zeta)}{z-\zeta},
\end{equation}
respectively. Here $G$ is defined and continuous in all $\C$ and
satisfies, in the sense of distributions,
$$
\frac{\partial G}{\partial\bar{z}}= \overline{g}g\chi_{\D}.
$$
Thus, in $\D$,
$$
G=\bar{f}g+H
$$
for some $H\in\mathcal{O}({\mathbb{D}})$. This equality also defines
$H$ on $\partial\D$, by which it becomes continuous on
$\overline{\D}$.

On the other hand, (\ref{sigma}) shows that $\hat{\sigma}=G$ outside
$\overline\D$, and by continuity this also holds on $\partial\D$.
Hence
$$
f^*g=\bar{f}g=G-H =\hat{\sigma}-H
$$
on $\partial\D$, and since the right member is holomorphic in
$\D\setminus \overline{\D(0,\rho)}$ the desired meromorphic
extension of $f$ follows.

\end{proof}

If $f(\cdot,t)\in \mathcal{O}_{\rm univ}(\overline{\mathbb{D}})$ is
a univalent weak solution then it is known
\cite{Gustafsson-Prokhorov-Vasiliev-2004}, \cite{Gustafsson-Vasiliev-2006}
that the radius of analyticity of $f$ is an increasing function of
time. In the non-univalent case this is no longer true, but there is
a related radius (essentially $1/r$ in the previous lemma) which is stable in
time (actually increases), and this will be a good enough
statement for our needs.

\begin{lemma}\label{lem:radius1}
Let $\tilde{\Omega}(\cdot,t)=\tilde{f}(\D,t)$ be a simply connected weak solution
on a Riemann surface $\mathcal{M}$ with projection $p:\mathcal{M}\to\C$ and let $f(\zeta,t)=p(\tilde{f}(\zeta,t))$.
Assume $q(t)\geq 0$ and that for a certain $0<r<1$ the equivalent conditions in
Lemma~\ref{lem:radius} hold for $f=f(\cdot, 0)$. Then they hold with
the same $r$ for all $f(\cdot,t)$, $t>0$.
\end{lemma}

\begin{proof}
If  $f(\cdot,t)\in \mathcal{O}_{\rm norm}(\overline{\mathbb{D}})$ is
a weak solution on $\mathcal{M}$ starting at $t=0$ then, by
(\ref{weakunitdisk2}),
$$
\int_{\D} h (z)|g(z,t)|^2 \,dm(z) =\int_{\D} h
(\varphi(\zeta,0,t))|g(\zeta,0)|^2 \,dm(\zeta)+2\pi Q(t)h(0)
$$
for all $h\in\mathcal{O}(\overline\D)$. Assume now that condition
$(ii)$ of Lemma~\ref{lem:radius} holds at $t=0$ for some $0<r<1$.
Since $|\varphi(\zeta,0,t)|\leq |\zeta|$ by Schwarz' lemma we then
get, for an arbitrary $\rho$ with $r<\rho<1$,
$$
|\int_{\D} h (\varphi(\zeta,0,t))|g(\zeta,0)|^2dm(\zeta)| \leq C_\rho
\sup_{\zeta\in\D(0,\rho)}|h (\varphi(\zeta,0,t))| \leq C_\rho
\sup_{z\in\D(0,\rho)}|h (z)|,
$$
hence
$$
|\int_{\D} h (z)|g(z,t)|^2 \,dm(z)|\leq (C_\rho+2\pi
Q(t))\sup_{z\in\D(0,\rho)}|h (z)|.
$$
This shows that $(ii)$ of Lemma~\ref{lem:radius} holds also at any time
$t>0$, with the same $r$ as for $t=0$, which is what we needed to prove.

\end{proof}

The final auxiliary result is a conjecture, which is very likely to be true but for which we still have no complete proof at present. It was therefore
was listed among the assumptions in Theorem~\ref{thm:universal}. It  concerns the issue of keeping $\tilde{\Omega}(t)$ simply connected
all the time. With the weak solution pulled back to $\D$, as in Example~\ref{ex:usageweaksol},
the crucial statement becomes the following, formulated in terms of (\ref{weaksolutionexample}).

\begin{conj}\label{lem:simplyconnected}
Let $g\in\mathcal{O}(\overline{\D})$ be fixed (independent of $t$)
and denote by $\{D(t): 0\leq
t<\varepsilon\}$ the weak solution for the weight $|g|^2$ and
initial domain $D(0)=\D$, with $\varepsilon>0$ is so small
that all domains $D(t)$ are compactly contained in the region of analyticity
of $g$. In other words,
$$
\int_{D(t)}h|g|^2dm \geq \int_{\D}h|g|^2dm+2\pi Q(t)h(0)
$$
for every $h\in SL^1(D(t),|g|^2m)$, and $0\leq t\leq \varepsilon$. Then, if $\varepsilon >0$ is sufficiently
small, the domains $D(t)$ are star-shaped with respect to the origin, in
particular simply connected.
\end{conj}

In terms of partial balayage, $D(t)\supset\D$ is given by
$$
\text{Bal\,}(2\pi Q(t)\delta_0, |g|^2\chi_{G\setminus\D}m)=|g|^2\chi_{D(t)\setminus\D}m,
$$
where $G\supset \overline{\D}$ is the domain of analyticity of $g$.
If $g$ has no zeros on $\partial \D$, then Conjecture~\ref{lem:simplyconnected} 
indeed holds, by virtue of stability results for free boundaries \cite{Caffarelli-1981}, \cite{Friedman-1982},
or else by existence of classical solutions \cite{Escher-Simonett-1997}. 
But we need Conjecture~\ref{lem:simplyconnected}
exactly in the case when $g$ has zeros on $\partial\D$.
 
{\it Some steps towards a proof of Conjecture~\ref{lem:simplyconnected}.}
In terms of the function $u=u(z,t)$ appearing in Definition \ref{def:partialbalayage} for the choice
$\mu=2\pi Q(t)\delta_0$, $\lambda=|g|^2\chi_{G\setminus \D} m$, the weak solution
$\{D(t):0\leq t<\varepsilon\}$ is given by
$$
D(t)=\{z\in\C:u(z,t)>0\}
$$
with $u$ satisfying (and determined by)
$$
\begin{cases}
u\geq 0 \quad \text{in}\quad \C,\\
\Delta u = |g|^2\chi_{D(t)\setminus\D} -2\pi Q(t)\delta_0 \quad \text{in}\quad \C,\\
u=|\nabla u|=0 \quad \text{outside}\quad D(t).
\end{cases}
$$
These properties follow in a standard manner from the complementarity
system (of  type (\ref{twoineq}), (\ref{compl})) satisfied by $u$.

Now write, in terms of polar coordinates $z=re^{\I\theta}$,
$$
v=r\frac{\partial u}{\partial r}.
$$
This function is continuous in $\C\setminus\{0\}$
because the elliptic partial differential equation which $u$ satisfies (in the
sense of distributions) shows that $u$ is continuously differentiable,
even across $\partial D(t)$. In order to show that
$D(t)$ is star-shaped it is enough to show that $u$ decreases in each radial direction, i.e., that $v\leq 0$. This
is what one hopes to show, for $t>0$ small enough.

In the region $D(t)\setminus\D$ we have $\Delta u=|g|^2$, and easy
computations show that this translates into the equation
\begin{equation}\label{Deltav}
\Delta v=2|g|^2\re(1+\frac{zg'}{g}), \quad z\in D(t)\setminus\D,
\end{equation}
for $v$. As to boundary conditions we have
$$
v=0 \quad \text{on}\quad \partial D(t)
$$
since $u$ vanishes together with its first derivative there.
Inside $\D$, $u(z)=-Q(t)\log|z|+\text{harmonic}$, and except for the logarithmic singularity at
the origin, $u$ is continuously differentiable in all $D(t)$. 
It follows that $v$ is harmonic in $\D$ with $v(0)=-Q(t)$ and that $v$ is continuous in all $D(t)$.
(On $\partial\D$ there is a jump in the first derivatives.)

Unfortunately we cannot be sure of the sign of the right member of (\ref{Deltav}).
If we knew that it was nonnegative, then
the desired conclusion $v\leq 0$ would follow immediately from the maximum
principle. To clarify the situation we make a local analysis around a point on $\partial\D$ at which $g$
vanishes. We may assume that this point is $z=1$, and then we can write
$$
g(z)=(z-1)^d h(z),
$$
where $d$ is the order of the zero and $h$ is analytic with $h(1)\ne 0$.
We then compute the right member of (\ref{Deltav}) as
$$
2|g(z)|^2\re(1+\frac{zg'(z)}{g(z)})=2|z-1|^{2d}|h(z)|^2\re(1+\frac{dz}{z-1}+\frac{zh'(z)}{h(z)})
$$
$$
=2|z-1|^{2d}\re(zh'(z)\overline{h(z)})+2|z-1|^{2(d-1)}|h(z)|^2\re((z-1)(\bar z-1)+d\cdot z(\bar z-1))
$$
$$
=2|z-1|^{2(d-1)}|h(z)|^2 
\left(|z-1|^2\re\frac{zh'(z)}{h(z)}
+(d+1)(|z-\frac{d+2}{2d+2}|^2-(\frac{d}{2d+2})^2\right).
$$
Here the second term inside the bracket is positive outside the circle with center $\frac{d+2}{2d+2}$ and radius $\frac{d}{2d+2}$, 
in particular outside $\D$, while the first term may have any sign.

Thus,  the right member in (\ref{Deltav}) is positive in major parts of neighborhoods (outside $\D$) of 
points on $\partial\D$ where $g$ vanishes. 
In some examples, like if $h$ is constant, which will be the case in the example in Section~\ref{sec:cardioid3} below,
it follows that the right member in (\ref{Deltav}) is positive in all $D(t)\setminus\D$, and the star-shapedness can
be inferred.  
Close to other points on $\partial\D$ one can perform an analysis based on
known stability behavior of free boundaries  \cite{Caffarelli-1981}, \cite{Friedman-1982}.
This gives at least that, outside any fixed neighborhood of $z=1$, $D(t)$ does not have any holes if $t>0$ is small
enough.

\begin{proof} (of theorem)

To get started, observe that we can find $\mathcal{M}$ so that $(i)$ holds. It is just to take $\mathcal{M}=\D$,
$\tilde{f}(\zeta,0)=\zeta$ and $p=f(\cdot,0)$. Compare the proof of
Lemma~\ref{lem:riemannsurface}. The remaining part of the proof
consists of extending the Riemann surface $\mathcal{M}$ so that $(ii)$ remains
valid; $(i)$ will automatically remain valid.

So assume that we have constructed $\mathcal{M}$ so that $(ii)$ holds on a
time interval $[0,T]$, where $T\geq 0$ ($T=0$ not excluded). We
shall show how to extend $\mathcal{M}$ (if necessary) and the solution, to
some interval $[0,T+\varepsilon]$, $\varepsilon>0$.

We have $\tilde{\Omega}(t)=\tilde{f}(\D,t)$, where
$\tilde{f}(\cdot,t):\D\to \mathcal{M}$,
$f(\cdot,t)=p\circ\tilde{f}(\cdot,t)\in \mathcal{O}_{\rm
norm}(\overline{\D})$, $f(\cdot,t)$ is a subordination chain and
$f(\cdot,t)$ is meromorphic in a disk $\D(0,\frac{1}{r})$, with
$0<r<1$ independent of $t$ by Lemma~\ref{lem:radius1}.

Set $\mathcal{M}_T=\tilde{\Omega}(T)\subset \mathcal{M}$. This is the only part of $\mathcal{M}$
which is needed up to time $T$, and it can be identified with
$\D_T=\D$ via $\tilde{f}(\cdot,T)$. Now choose $1<\rho< \frac{1}{r}$
(with $r$ as in Lemma~\ref{lem:radius1}) so that $f'(\zeta,T)$ has
no zeros for $1<|\zeta|<\rho$ (but may have it for $|\zeta|=1$).
Viewing $\D(0,\rho)\supset\D$ as a Riemann surface over $\C$ with
covering map $f(\cdot,T)$ we get, on the level of abstract Riemann
surfaces, an extension $\mathcal{M}'\supset \mathcal{M}$ of $\mathcal{M}$. On $\mathcal{M}'$ we can continue
the weak solution to some time interval $[0,T+\varepsilon]$,
$\varepsilon>0$. Compare the discussion in Example~\ref{ex:usageweaksol}. Assuming $\varepsilon>0$ is small enough this
solution $\tilde{\Omega}(t)$ remains simply connected, assuming
Conjecture~\ref{lem:simplyconnected}. Then set
$\mathcal{M}_{T+\varepsilon}=\tilde{\Omega}(T+\varepsilon)$.

Thus we can always extend a weak solution defined on a closed time
interval to a larger interval. We also have to show that whenever we
have a solution on a half-open interval $[0,T)$ (with $T>0$) it can
be extended to the closure $[0,T]$. However, this is fairly
immediate because we can simply define
$\tilde{\Omega}(T)=\mathcal{M}_T=\cup_{0\leq t<T} \mathcal{M}_t$ (cf. proof of
Lemma~\ref{lem:riemannsurface}). This surface is easily seen to be
simply connected (because any closed curve in $\mathcal{M}_T$ will lie
entirely in $\mathcal{M}_t$ for some $t<T$). Moreover, the defining property
(\ref{weaksolutionR}) of a weak solution will hold on all
$[0,T]$, and since the radius of analyticity of $f(\cdot,T):\D\to
\tilde{\Omega}(T)$ is larger than one (Lemma~\ref{lem:radius1}),
$\tilde{\Omega}(T)\cong \D$ will have compact closure in a larger
Riemann surface $\mathcal{M}\supset \mathcal{M}_T$, on which the evolution may continue.

The above arguments show that there is no finite stopping time for
the construction of $\mathcal{M}$ and a simply connected weak solution in $\mathcal{M}$.
Therefore part $(ii)$ of the theorem follows.

Assertion $(iii)$ of the theorem is an easy consequence of Proposition~\ref{prop:compatibility}.
\end{proof}


\section{Example: several evolutions of a cardioid}\label{sec:examples}

In order to illustrate Theorem~\ref{thm:universal}, as well as some forthcoming results, in particular
Theorem~\ref{thm:rational}, we shall consider three different Hele-Shaw evolutions which all start out from the cardioid
$\Omega(0)=f(\D,0)$, where
\begin{equation}\label{polycardoid}
f(\zeta,0)=\zeta -\frac{1}{2}\zeta^2.
\end{equation}
Thus $g(\zeta,0)=1-\zeta$, $\omega_1(0)=1$, and there is
a cusp at the point $f(1,0)=\frac{1}{2}$ on $\partial\Omega(0)$.
It is a major open problem to find some natural way to make Hele-Shaw suction ($q<0$) starting from the above cardioid,
and we have not made much progress on that, except for a minor remark at the end of the example in Section~\ref{sec:cardioid3}. 
In essence, our three solutions will all correspond to injection ($q>0$), and two
of them will be non-univalent. 


\subsection{The univalent solution}\label{sec:cardioid1}
This is the ordinary univalent Hele-Shaw evolution $f(\cdot,t)\in\mathcal{O}_{\rm univ}(\overline{\D})$,
which by conservation of $M_1=a_1^2\bar{a}_2=-\frac{1}{2}$ is given by
$$
f(\zeta,t)=a_1(t)\zeta+a_2(t)\zeta^2= a_1(t)\zeta-\frac{1}{2a_1(t)^2}\zeta^2.
$$
Adapting $q(t)$ so that $a_1(t)=e^t$ ($0\leq t<\infty$), for example, gives
$$
f(\zeta,t)=e^t\zeta-\frac{1}{2}e^{-2t}\zeta^2, \quad q(t)=e^{2t}-e^{-4t}.
$$
Note that $\omega_1(t)=e^{3t}$ starts out with finite speed, despite the cusp. This is possible because $q(0)=0$.
For $t>0$, $q(t)>0$.


\subsection{A non-univalent solution of the Polubarinova-Galin equation}\label{sec:cardioid2}
In the univalent solution, the coefficient $a_1$ ranges over the interval $1\leq a_1<\infty$, and the moment 
$M_0=a_1^2+2|a_2|^2=a_1^2+\frac{1}{2}a_1^{-4}$
is an increasing function of $a_1$. But, as a function of $a_1$,  $M_0$ is strictly convex on the entire interval $0<a_1<\infty$, and 
it has a minimum for $a_1=1$. Thus $M_0$ increases also as $a_1$ decreases from $1$ to $0$. Choosing then $a_1=e^{-t}$,
$0\leq t<\infty$, gives our second solution
$$
f(\zeta,t)=e^{-t}\zeta-\frac{1}{2}e^{2t}\zeta^2, \quad q(t)=e^{4t}-e^{-2t}.
$$
This is not even locally univalent, but it does solve the Polubarinova-Galin equation. The zero of $g(\zeta,t)$, $\omega_1(t)=e^{-3t}$,
moves from the unit circle towards the origin, and its image point, $f(e^{-3t},t)=\frac{1}{2}e^{-4t}$ also moves. Therefore the solution
cannot be lifted to a fixed Riemann surface, and $f(\zeta,t)$ does not solve the L\"owner-Kufarev equation (see Theorem \ref{lkpg}).


\subsection{A non-univalent solution of the L\"owner-Kufarev equation}\label{sec:cardioid3}

Even though the first, univalent, solution is perfectly good in all respects, the solution which is
constructed in the proof of Theorem~\ref{thm:universal} is a different one, namely one
which goes up on a Riemann surface with two sheets. This is because the solution in the proof
is constructed in such a way that at any time, say $t=t_0$, at which a zero of $g(\zeta,t)$
reaches $\partial\D$, the continued solution propagates on the Riemann surface which
uniformizes $f^{-1}(\zeta,t_0)$ in a neighborhood of $\overline{\D}$.
This is a necessary step in most cases, but occasionally (as in the present
example, with $t_0=0$) it turns out that the original Riemann surface was actually good enough.

Below we calculate that solution which the proof of Theorem~\ref{thm:universal} would have given us.
This has the additional advantage of giving a reference solution which can be used as comparison in
order to obtain estimates for other solutions.
The idea (see also Section~\ref{sec:rational}) is that the initial $g$, which we write as
$$
g(\zeta,0)=(1-\zeta)\,\frac{(\zeta-1)(\zeta-1)}{(\zeta-1)^2},
$$
continues as
$$
g(\zeta,t)= b(t)\,\frac{(\zeta-\omega_1(t))(\zeta-\omega_2(t))(\zeta-\omega_3(t))}{(\zeta-\zeta_1(t))^2},
$$
where one of the zeros, say $\omega_1(t)$, moves into $\D$, $\zeta_1(t)=\omega_1^*(t)$ and where $\omega_2(t)$,
$\omega_3(t)$ in addition are chosen so that $g(\zeta,t)$ has no residues. This means that $f(\zeta,t)$ will be of the form
\begin{equation}\label{fcardoid}
f(\zeta,t)=-\frac{b_1\zeta+b_2\zeta^2+b_3\zeta^3}{\zeta-\zeta_1}
\end{equation}
with $b_1=b_1(t)$, $b_2=b_2(t)$, $b_3=b_3(t)$ and $\zeta_1=\zeta(t)$ all real. The parameters $b_1$ and $\zeta_1$
will turn out to be positive and strictly increasing in time. At time $t=0$ we have
\begin{equation}\label{bk}
\begin{cases}
b(0)=-1,\\
b_1(0)=1,\\
b_2(0)=-\frac{3}{2},\\
b_3(0)=\frac{1}{2},\\
\omega_1(0)=\omega_2(0)=\omega_3(0)=\zeta_1(0)=1.
\end{cases}
\end{equation}

From (\ref{fcardoid}) we obtain
\begin{equation}\label{gfrac}
g(\zeta,t)=\frac{b_1 \zeta_1+2b_2 \zeta_1\zeta-(b_2-3b_3\zeta_1)\zeta^2-2b_3\zeta^3}{(\zeta-\zeta_1)^2},
\end{equation}
$$
f^*(\zeta,t)=\frac{b_1\zeta^2+b_2\zeta+b_3}{\zeta^2(\zeta_1\zeta-1)}.
$$
The coefficients $b_j=b_j(t)$ and the pole $\zeta_1=\zeta_1(t)$ are to be determined according
to the following principles:

\begin{itemize}

\item The reflected point of $\zeta_1(t)$ is to be a zero of $g$:
$$
g(1/\zeta_1(t),t)=0.
$$

\item $f(\cdot ,t)$ shall map the above point $1/\zeta_1(t)$ to a point which does not move:
$$
f(1/\zeta_1(t),t)={\rm constant}=f(1,0)=\frac{1}{2}.
$$

\item The moment $M_1(t)$ is conserved in time:
$$
M_1(t)=\res_{\zeta=0}(ff^*gd\zeta)= M_1(0)=-\frac{1}{2}.
$$

\item $M_0(t)$ evolves according to
$$
M_0(t)=\res_{\zeta=0}(f^*gd\zeta)=M_0(0)+2Q(t)=\frac{3}{2}+2Q(t).
$$

\end{itemize}

The constant values $\pm \frac{1}{2}$ and $\frac{3}{2}$ above are obtained from the initial data
(\ref{bk}). Spelling out, the above equations become
$$
\begin{cases}
b_1 \zeta_1^4 +2b_2 \zeta_1^3- (b_2-3b_3\zeta_1)\zeta_1- 2b_3=0,\\
b_1\zeta_1^2+b_2 \zeta_1+b_3 +\frac{1}{2}\zeta_1^2(1-\zeta_1^2)=0,\\
b_1^2b_3 -\frac{1}{2}\zeta_1^2=0,\\
b_1b_2\zeta_1+2b_2b_3\zeta_1+b_1b_3(\zeta_1^2+2)+(\frac{3}{2}+2Q)\zeta_1^2=0.
\end{cases}
$$
Here we have four equations for the five time dependent parameters $b_1$, $b_2$, $b_3$, $\zeta_1$ and $Q$.
It turns out that it is possible to solve this system by expressing all paramenters in terms of $b_1$:
$$
\begin{cases}
\zeta_1=+\sqrt{\frac{1}{2}(1+2b_1-\frac{1}{b_1^2})},\\
b_2= -\frac{\zeta_1}{4}(1+2b_1 +\frac{3}{b_1^2}),\\
b_3= \frac{2b_1^3+b_1^2-1}{4b_1^4},\\
Q=\frac{1}{16b_1^6}(4b_1^8+2b_1^7-12b_1^6+b_1^4+6b_1^3+2b_1^2-3).
\end{cases}
$$
The range for $\zeta_1=\zeta(t)$ is $1\leq \zeta_1<\infty$. At time $t=0$ we shall have $\zeta_1=1$.
Then also $b_1=1$, and since one easily checks that $\frac{d\zeta_1}{db_1}>0$ it is appropriate to fix
the time scale by setting
$$
b_1(t)=e^t.
$$
By this all parameters $b_1$, $b_2$, $b_3$, $\zeta_1$, $Q$ become explicit functions of $t$.
Including expansions for small $t>0$ we have
\begin{equation}\label{explicitcoefficients}
\begin{cases}
\zeta_1(t)=\sqrt{\frac{1}{2}(1+2e^t-e^{-2t})}=1+t-\frac{3}{4}t^2+O(t^3),\\
b_1(t) =e^t=1+t+\frac{1}{2}t^2+O(t^2),\\
b_2(t)=-\frac{1}{4\sqrt{2}} (1+2e^t+3e^{-2t}) \sqrt{1+2e^t-e^{-2t}}
= -\frac{3}{2}-\frac{1}{2}t+\frac{3}{8}t^2+O(t^3),\\
b_3(t)=\frac{1}{4}(2e^{-t}+e^{-2t}-e^{-4t})=\frac{1}{2}-\frac{5}{4}t^2+O(t^3),\\
Q(t)=\frac{1}{16}(4e^{2t}+2e^t-12+ e^{-2t}+6e^{-3t}+2e^{-4t}-3e^{-6t})= 4t^3+O(t^4).
\end{cases}
\end{equation}

This gives
\begin{equation}\label{explicitq}
q(t)=\frac{1}{8}(4e^{2t}+e^t-e^{-2t}-9e^{-3t}-4e^{-4t}+9e^{-6t})=12t^2+O(t^3),
\end{equation}
$$
f(\zeta,t)=-\frac{2(1+t)\zeta-(3+t)\zeta^2+\zeta^3+O(t^2)}{2(\zeta-1-t+O(t^2))}.
$$
Thus $q(0)=\dot{q}(0)=0$, while for $t>0$, $q(t)>0$, so the evolution is very slow in the beginning,
in fact so slow that it is not at all singular at $t=0$. 

As for $g(\zeta,t)$, we already know (by construction) that one of its zeros is
$\omega_1(t)=1/{\zeta_1(t)}$.
By dividing out this zero in (\ref{gfrac}) one gets $g$ on the form
$$
g(\zeta,t)=-2b_3(t) \,\frac{(\zeta-1/\zeta_1(t))(\zeta^2-\frac{1}{2}(b_1(t)^2+3)\zeta_1(t)\,\zeta+b_1(t)^3)}{(\zeta-\zeta_1(t))^2},
$$
and the remaining two zeros $\omega_2(t)$, $\omega_3(t)$ are the zeros of the second degree polynomial in the numerator.
One easily checks that the discriminant of that polynomial is negative on some interval $0<t<\varepsilon$, hence
$\omega_2(t)$, $\omega_3(t)$ are non-real (a complex conjugate pair) for those values of $t$. For large $t$ they are however real 
(the discriminant is positive). From
$$
\omega_2(t)+\omega_3(t)=\frac{1}{2}(b_1(t)^2+3)\zeta_1(t), \quad \omega_2(t)\omega_3(t)=b_1(t)^3
$$
one also realizes that the real parts of the two roots are increasing functions of $t$, for all $0<t<\infty$.

The solution $f(\zeta,t)$ represents  an evolution
of the cardioid which is non-univalent regarded as a map into $\C$ but which can be
viewed as a univalent map into a two-sheeted Riemann surface over $\C$. It is the solution
which comes out of the construction in the proof of Theorem~\ref{thm:universal}.
Conjecture~\ref{lem:simplyconnected} then concerns the solution pulled back to the unit disk by $f(\zeta,0)=\zeta-\frac{1}{2}\zeta^2$,
thus the function $g$ in that lemma is $g(\zeta)=f'(\zeta,0)=1-\zeta$. The inverse of $f(\zeta,0)$ is $f^{-1}(z,0)=1-\sqrt{1-2z}$,
hence one gets the function
$$
F(\zeta,t)=f^{-1}(f(\zeta,t),0)=1-\sqrt{1+\frac{2}{\zeta-\zeta_1(t)}(b_1(t)\zeta+b_2(t)\zeta^2+b_3(t)\zeta^3)},
$$
which, for $0<t<\varepsilon$ say, maps $\D$ conformally onto the slightly larger domain $D(t)$
(in the notation of Conjecture~\ref{lem:simplyconnected}). Because of the square root it is not entirely
trivial that $F(\zeta,t)$ is single-valued in $\D$. The pole at $\zeta=\zeta_1(t)$ causes no problem in this
respect since $\zeta_1(t)\notin\D$, but since $1\in D(t)=F(\D,t)$ there must be a point in $\D$ for which the expression under the square root vanishes. 
However, despite this the square root does in fact resolve into a single-valued function in $\D$.
In terms of the notations in Lemma~\ref{lem:riemannsurface} and Section~\ref{sec:lifting} we have $F(\zeta,t)=\tilde{f}(\zeta,t)$,
$f(\zeta,0)=p(\zeta)$, $g(\zeta)=p'(\zeta)$.

The domains $D(t)$ are star-shaped with respect to the origin, hence the solution exists for all $0<t<\infty$, cf.
\cite{Gustafsson-Prokhorov-Vasiliev-2004}.  Indeed, the star-shapedness
follows from equation (\ref{Deltav}) in the proof of Theorem~\ref{thm:universal}, which in the present context becomes
$$
\Delta v= 4(|z-\frac{3}{4} |^2-\frac{1}{16}), \quad z\in D(t)\setminus\D.
$$
Here, the right member is non-negative, and as $\Delta v=0$ in $\D$, $v$ is continuous in $D(t)$ and $v=0$ on $\partial D$,
the required inequality $v\leq 0$ in $D(t)$ follows from the maximum principle.

An interesting aspect is that the now fully explicit solution $f(\zeta,t)$, defined for $0<t<\infty$ by (\ref{fcardoid}), (\ref{explicitcoefficients}),
is not only smooth at $t=0$, it even has a real analytic continuation across $t=0$.
This extended solution, defined on $-\varepsilon<t<\infty$ (say), has the drawback that it has a pole inside $\D$ ($\zeta_1(t)\in\D$ 
when $t<0$), but $q(t)$ remains positive
for $t<0$, as can be seen from (\ref{explicitq}). This means that the solution represents suction out of the cardioid as $t$ decreases to negative values.


\section{Structure of rational solutions}\label{sec:rational}

In this section we shall prove that the property of $g=f'$ being
a rational function is preserved in time for weak solutions as long as they remain simply connected. 
In other words, it is preserved by the L\"owner-Kufarev equation, even under transition of zeros of $g$
through $\partial\D$.  However,  $g$ acquires additional zeros and poles
under such an event, and the transition will not be smooth.
We shall also show that for certain other solutions
of the Polubarinova-Galin equation rationality is also preserved.

We shall give two avenues to the question of rationality: first a direct approach, just making an `Ansatz' of a rational $g$ in a suitable version of the Polubarinova-Galin equation, and secondly
via quadrature identities,
which are related to the concept of a weak solution and which can incorporate transitions
of zeros through $\partial\D$.


\subsection{Direct approach}

We assume that $g$ rational of the form (\ref{structureg}),
and we address the question to which extent this form is preserved in time for
solutions of the Polubarinova-Galin or L\"owner-Kufarev equations when $g$ is
allowed to have zeros in $\D$.

Recall (Theorem~\ref{thm:R}) that the Polubarinova-Galin equation is equivalent
to a relaxed version of the L\"owner-Kufarev equation,
\begin{equation}\label{lkgen1}
\dot{f}(\zeta,t)=\zeta
f'(\zeta,t)\left(P(\zeta,t)+R(\zeta,t)\right),
\end{equation}
where $P(\zeta,t)$ is the Poisson integral (\ref{poisson}) and where
$R(\zeta,t)$ is any function of the form (\ref{definitionR}). We shall here assume, for simplicity, that
the zeros $\omega_j$ of $g$ are simple, and then (\ref{definitionR}) becomes
\begin{equation}\label{definitionR1}
R(\zeta,t)=\I\im\sum_{\omega_j\in\D} \frac{2B_j(t)}{\omega_j(t)}
+\sum_{\omega_j\in\D}\left( \frac{2B_j(t)}{\zeta -\omega_j(t)}-
\frac{2\overline{B_j(t)}\zeta}{1 -\overline{\omega_j(t)}\zeta}
\right).
\end{equation}

The equation (\ref{lkgen1}) for $f$ is equivalent to the equation, generalizing (\ref{PGg}),
\begin{equation}\label{PRGg}
\frac{\partial}{\partial t}(\log g)=\zeta (P+R)\frac{\partial}{\partial \zeta}(\log g)+\frac{\partial}{\partial \zeta}(\zeta (P+R))
\end{equation}
for $g$. Here the derivatives of $\log g$ are obtained from (\ref{loggdot}), (\ref{loggprime}), and it only
remains to evaluate the Poisson integral $P(\zeta,t)$. This can be done by a simple residue calculus in
(\ref{poisson}), using that $|g(\zeta,t)|^2=g(\zeta,t)g^*(\zeta,t)$
is a rational function in $\zeta$. However, the calculation becomes more
transparent if everything is done at an algebraic level, by which it
essentially reduces to an expansion in partial fractions.

Recall that, by definitions of $P$ and $R$,
\begin{equation}\label{PPstar}
P(\zeta,t)+P^*(\zeta,t)=\frac{2q(t)}{g(\zeta,t)g^*(\zeta,t)},
\end{equation}
$$
R(\zeta,t)+R^*(\zeta,t)=0.
$$
The rational function $q(t)/g(\zeta,t)g^*(\zeta,t)$ has
poles at the zeros of $g$ and $g^*$, i.e., at $\omega_1,
\dots,\omega_m, \omega_1^*, \dots,\omega_m^*$. At infinity it has
the behavior (by (\ref{structureg}))
\begin{align}\label{Ainfty}
\lim_{\zeta\to\infty}\frac{q(t)}{g(\zeta,t)g^*(\zeta,t)}
=A_\infty=\begin{cases}
\frac{q\prod_{j=1}^n\bar{\zeta}_j}{|b|^2\prod_{j=1}^m\bar{\omega}_j}
\quad &{\rm if\,\,} m=n, \\
0 \quad &{\rm if\,\,} m>n.
\end{cases}
\end{align}
We shall assume, in addition to the zeros $\omega_k$ being simple, that no two zeros are
reflections of each other with respect to the  unit circle, i.e., we assume that
$\omega_k\ne \omega_j^*$ for all $k,j$
and, in particular ($k=j$), that there are no zeros on the unit circle.
These assumptions are necessary in order to expect the existence of a smooth solution of
the Polubarinova-Galin equation (\ref{pg1}), and even more so for the L\"owner-Kufarev equation.
Indeed, spelling out (\ref{pg1}) as an identity between rational functions as
$$
\dot{f}(\zeta,t)\cdot\zeta^{-1}g^*(\zeta,t) +\dot{f}^*(\zeta,t)\cdot\zeta g(\zeta,t)=2q(t)
$$
we see that if, for some particular value of $t$, $g$ and $g^*$ have a common zero, with $\dot{f}$ and $\dot{f}^*$
finite, then $q(t)$ must be zero. 

With the above  assumptions in force we can write
\begin{equation}\label{partialfractions}
\frac{q(t)}{g(\zeta,t)g^*(\zeta,t)} =A_\infty+\sum_{k=1}^m
\frac{\overline{A}_k}{\overline{\omega}_k}+\sum_{k=1}^m
\left[\frac{A_k}{\zeta-\omega_k}
+{\frac{\overline{A}_k\zeta}{1-\overline{\omega}_k\zeta}}\right],
\end{equation}
where the coefficients
$A_k=A_k(t,b,\omega_1,\dots,\omega_m,\zeta_1,\dots,\zeta_n)$ are
given by
\begin{equation}\label{Ak}
A_k=\frac{q(t)}{g'(\omega_k,t)g^*(\omega_k,t)}
=\frac{q}{|b|^2}\cdot\frac{\prod_{j}(\omega_{k}-\zeta_{j})\prod_{j}\overline{(\omega_{k}^{*}-\zeta_{j})}}
{\prod_{j\neq
k}(\omega_{k}-\omega_{j})\prod_{j}\overline{(\omega_{k}^{*}-\omega_{j})}}
\end{equation}
for $1\leq k\leq m$.
Notice that some of the $A_k$ may vanish: if $\omega_k\in\D$
and $\omega_k^*$ coincide with one of the poles $\zeta_j$, then $A_k=0$.

Now, $P(\zeta,t)$ is to be that holomorphic
function in $\mathbb{D}$ whose real part has boundary values
$q(t)/g(\zeta,t)g^*(\zeta,t)$ and whose imaginary part vanishes at
the origin. The function (\ref{partialfractions}) itself certainly
has the right boundary behaviour on $\partial\mathbb{D}$, but it is
not holomorphic in $\mathbb{D}$. On the other hand, the two types of
polar parts occurring in (\ref{partialfractions}) have the same real
parts on the boundary:
$$
{\rm Re\,}\frac{A_k}{\zeta-\omega_k}={\rm
Re\,}{\frac{\overline{A}_k\zeta}{1-\overline{\omega}_k\zeta}} \quad
{\rm on\,\,}\partial\mathbb{D}.
$$
Therefore, without changing the real part on the boundary we can
make the function (\ref{partialfractions}) holomorphic in
$\mathbb{D}$ by a simple exchange of polar parts. In addition, one
can add a purely imaginary constant to account for the normalization
of $P$ at the origin. This gives
\begin{equation}\label{structureP}
P(\zeta,t)=A_0+\sum_{\omega_j\in\C\setminus\overline{\D}}\frac{2A_j}{\zeta-\omega_j}
+\sum_{\omega_j\in\D} \frac{2\overline{A}_j\zeta}{1-\overline{\omega}_j\zeta},
\end{equation}
with the $A_j=A_j(t)$ given by (\ref{Ak}) for $1\leq j\leq m$. For $A_0$ we have
$$
{\rm Re\,}A_0={\rm Re\,}A_\infty+{\rm Re\,}\sum_{k=1}^m
\frac{{A}_k}{{\omega}_k},
$$
$$
{\rm Im\,}A_0= {\rm Im\,}\sum_{\omega_j\in\C\setminus\overline{\D}}\frac{2A_j}{\omega_j},
$$
so that the real part of (\ref{partialfractions}) remains unaffected in the passage to (\ref{structureP}),
and so that the normalization ${\rm Im\,}P(0,t)=0$ is achieved.
Note that $\re(P(\zeta,t)+R(\zeta,t))\geq 0$ in
$\D$ if and only if $R=0$ (because if $R\ne 0$ then $R$ has poles in
$\D$).

Thus
$$
P(\zeta,t)+R(\zeta,t)=
$$
$$
=A_0+\I\im\sum_{\omega_j\in\D} \frac{2B_j}{\omega_j}
+\sum_{\omega_j\in\C\setminus\overline{\D}}\frac{2A_j}{\zeta-\omega_j}
+\sum_{\omega_j\in{\D}}\frac{2B_j}{\zeta-\omega_j}
+\sum_{\omega_j\in\D} \frac{2(\overline{A}_j-\overline{B}_j)\zeta}{1-\overline{\omega}_j\zeta},
$$
$$
=C+\sum_{\omega_j\in\C\setminus\overline{\D}}\frac{2A_j}{\zeta-\omega_j}
+\sum_{\omega_j\in{\D}}\frac{2B_j}{\zeta-\omega_j}
-\sum_{\omega_j\in\D} \frac{2(\overline{A}_j-\overline{B}_j)(\omega_j^*)^2}{\zeta-{\omega}_j^*},
$$
where
$$
C=A_0+\I\im\sum_{\omega_j\in\D} \frac{2B_j}{\omega_j}
-\sum_{\omega_j\in\D}{2(\overline{A}_j-\overline{B}_j)\omega_j^*}.
$$
Also,
$$
\zeta(P(\zeta,t)+R(\zeta,t))=
$$
$$
= C\zeta +D+\sum_{\omega_j\in\C\setminus\overline{\D}}\frac{2A_j\omega_j}{\zeta-\omega_j}
+\sum_{\omega_j\in{\D}}\frac{2B_j\omega_j}{\zeta-\omega_j}
-\sum_{\omega_j\in\D} \frac{2(\overline{A}_j-\overline{B}_j)(\omega_j^*)^3}{\zeta-{\omega}_j^*},
$$
with
$$
D=\sum_{\omega_j\in\C\setminus\overline{\D}}2A_j +\sum_{\omega_j\in{\D}}2B_j
-\sum_{\omega_j\in\D}{2(\overline{A}_j-\overline{B}_j)(\omega_j^*)^2}.
$$

In view of (\ref{logg}), (\ref{loggdot}), (\ref{loggprime})
the dynamical law (\ref{PRGg}) becomes
\begin{equation}\label{PGg1}
\frac{\dot b}{b}-\sum_{k=1}^m\frac{\dot\omega_k}{\zeta-\omega_k}
+\sum_{j=1}^n\frac{\dot\zeta_j}{\zeta-\zeta_j}=
\end{equation}
$$
=\big(C\zeta +D+\sum_{\omega_j\in\C\setminus\overline{\D}}\frac{2A_j\omega_j}{\zeta-\omega_j}
+\sum_{\omega_j\in{\D}}\frac{2B_j\omega_j}{\zeta-\omega_j}
-\sum_{\omega_j\in\D} \frac{2(\overline{A}_j-\overline{B}_j)(\omega_j^*)^3}{\zeta-{\omega}_j^*}\big)
$$
$$
\cdot\big(\sum_{k=1}^m\frac{1}{\zeta-\omega_k}
-\sum_{j=1}^n\frac{1}{\zeta-\zeta_j}\big)+
$$
$$
+C-\sum_{\omega_j\in\C\setminus\overline{\D}}\frac{2A_j\omega_j}{(\zeta-\omega_j)^2}
-\sum_{\omega_j\in{\D}}\frac{2B_j\omega_j}{(\zeta-\omega_j)^2}
+\sum_{\omega_j\in\D} \frac{2(\overline{A}_j-\overline{B}_j)(\omega_j^*)^3}{(\zeta-{\omega}_j^*)^2}.
$$

The derivatives $\dot b$, $\dot{\omega}_k$, $\dot{\zeta}_j$ to be determined appear as
coefficients in the constant term and poles of order one. Therefore (\ref{PGg1}) can
be satisfied only if all terms with poles of higher order cancel out. This automatically
occurs for the terms of the form $\frac{2A_j\omega_j}{(\zeta-\omega_j)^2}$ (${\omega_j\in\C\setminus\overline{\D}}$) and $\frac{2B_j\omega_j}{(\zeta-\omega_j)^2}$
(${\omega_j\in{\D}}$).

In order that the remaining terms
$\sum_{\omega_j\in\D} \frac{2(\overline{A}_j-\overline{B}_j)(\omega_j^*)^3}{(\zeta-{\omega}_j^*)^2}$
with poles of the second order shall disappear we must have, for each $j$ with $\omega_j\in\D$, that
\begin{equation}\label{AB}
A_j=B_j.
\end{equation}
These second order poles cannot cancel in any other way.
In order to allow the rational form (\ref{structureg}) to be stable in time we therefore from now on that  (\ref{AB}) holds.

Note that $R$ under this assumption becomes uniquely determined.
When (\ref{AB}) holds,
\begin{equation}\label{C}
C=A_0+\I\im\sum_{\omega_j\in\D} \frac{2A_j}{\omega_j}
={\rm Re\,}A_\infty+{\rm Re\,}\sum_{k=1}^m
\frac{{A}_k}{{\omega}_k}
+\I \,{\rm Im\,}\sum_{j=1}^m\frac{2A_j}{\omega_j},
\end{equation}
\begin{equation}\label{D}
D=\sum_{j=1}^m 2A_j
\end{equation}
and $P+R$ takes the simpler form
$$
P(\zeta,t)+R(\zeta,t)=C+\sum_{j=1}^m \frac{2A_j}{\zeta-\omega_j}.
$$
The dynamical law (\ref{PGg1}) now becomes
\begin{equation}\label{PGg2}
\frac{\dot b}{b}-\sum_{k=1}^m\frac{\dot\omega_k}{\zeta-\omega_k}
+\sum_{j=1}^n\frac{\dot\zeta_j}{\zeta-\zeta_j}=
\end{equation}
$$
=\big(C\zeta +D+\sum_{j=1}^m\frac{2A_j\omega_j}{\zeta-\omega_j}\big)
\cdot\big(\sum_{k=1}^m\frac{1}{\zeta-\omega_k}
-\sum_{j=1}^n\frac{1}{\zeta-\zeta_j}\big)
+C-\sum_{j=1}^m\frac{2A_j\omega_j}{(\zeta-\omega_j)^2}.
$$
Therefore (\ref{PGg1}) results in the following system
of ordinary differential equations for $\omega_k$, $\zeta_j$, $b$.

\begin{theorem}\label{thm:dynamics}

Under the assumption that $g$ has only simple zeros, that
$g$ and $g^* $ have no common zeros (in particular $g$ has no zero on $\partial \D$),
and that in addition (\ref{AB}) holds, the Polubarinova-Galin equation (\ref{pg1}), or (\ref{PRGg}), gives the following rational dynamics for $g$:
$$
\frac{d}{dt}\log\omega_k
=-C-\frac{2A_k}{\omega_k}-\sum_{j=1,\,j\ne k}^m \frac{2(A_k+A_j)}{\omega_k-\omega_j}+\sum_{j=1}^n\frac{2A_k}{\omega_k-\zeta_j}
$$
\begin{equation}\label{rationalzerosPR}
=P^*(\omega_k)+R^*(\omega_k)-\frac{2A_k}{\omega_k}(1+\sum_{j=1}^m
\frac{1}{1-\overline{\omega}_j \omega_k}-\sum_{j=1}^n
\frac{1}{1-\overline{\zeta}_j\omega_k}),
\end{equation}
\begin{equation}\label{rationalpolesPR}
\frac{d}{dt}\log \zeta_j
=-C-\sum_{k=1}^m
\frac{2A_k}{\zeta_j-\omega_k}
=P^*(\zeta_j)+R^*(\zeta_j),
\end{equation}
\begin{equation}\label{rationalbPR}
\frac{d}{dt}\log b=(m-n+1)C.
\end{equation}
Here the coefficients $A_j$, $C$ are given by (\ref{Ainfty}), (\ref{Ak}), (\ref{C}).

\end{theorem}

Since the $B_j(t)$ are completely free functions we can simply define them by (\ref{AB}).
Then (\ref{rationalzerosPR})--(\ref{rationalbPR}) is a regular system of ordinary differential equations,
having a unique solution as long as the stated conditions on the zeros of $g$ and $g^*$ hold.

The above unique rational solution of the Polubariova-Galin equation solves the L\"owner-Kufarev equation if
and only if $R=0$. By (\ref{AB}) this requires that $A_k=0$ for each $k$ with $\omega_k\in\D$.
Looking at (\ref{Ak}) we see that, if $q\ne 0$, the only way that $A_k$ can vanish at a given instant is that $\omega_k^*=\zeta_j$
for some $j$. However, it will be seen in Example~\ref{ex:onlyPG} below
that $A_k$ may vanish for some particular value of $t$ without vanishing for all $t$.
Therefore we also need that
$$
\dot{\omega}_k^*=\dot{\zeta}_j,
$$
so that the condition $A_k=0$ persists in time. We shall show that this is the case
if and only if $\omega_k^*$ is a pole of $g$ of order at least two (more generally,
of strictly higher order than that of the zero).

Assume that at one particular moment, say $t=0$, $\omega_k^*=\zeta_j$
for some pair $k$, $j$. Then $A_k=0$ at that moment so that (\ref{rationalzerosPR}),
(\ref{rationalpolesPR}) (with $R=0$) become
\begin{equation}\label{omegazetaPPstar}
\begin{cases}
\frac{d}{dt}\log\omega_k=P^*(\omega_k),\\
\frac{d}{dt}\log \zeta_j=P^*(\zeta_j).
\end{cases}
\end{equation}
This can also be written
$$
\begin{cases}
\frac{d}{dt}\log\omega_k^*=-P(\omega_k^*),\\
\frac{d}{dt}\log \zeta_j=P^*(\omega_k^*).
\end{cases}
$$
Thus we see that $\frac{d}{dt}\log\omega_k^*=\frac{d}{dt}\log \zeta_j$ holds
if and only if $P(\omega_k^*)+P^*(\omega_k^*)=0$, which, recalling (\ref{PPstar}),
happens if and only if $gg^*$ has a pole at $\omega_k^*=\zeta_j$.
Looking at the expression (\ref{structureg}) for $g$ one sees that this
occurs if and only if the pole of $g$ at $\zeta_j$ is of higher order than the zero of
$g$ at $\omega_k$.

Finally a remark about multiple zeros. If a zero $\omega_k\in \C\setminus \overline{\D}$
is of order $\geq 2$ then $P$ will have a pole at $\omega_k$ of the same order, and it is
easy to see that this pole will never cancel out in the dynamical equation (\ref{PGg1}). For this reason multiple
zeros outside $\overline{\D}$ can never survive, even though collisions may occur. The solution will
remain smooth over a collision because if two roots, $\omega_1$ and $\omega_2$, collide
the equations  still will be regular when  reformulated in terms of the combinations $\omega_1+\omega_2$
and $\omega_1\omega_2$.

On the other hand, if $g$ has a multiple zero $\omega_k$ in $\D$ this will not cause any higher order pole of
$P$ if $g$ has a pole of at least the same order at $\omega_k^*$, and if the order of the pole is of strictly higher
order then the same situation will persist in time. Therefore a solution will be obtained as before.

By now we have proved the following theorem on local behavior of solutions
of (\ref{lkgen1}), or (\ref{PRGg}).

\begin{theorem}\label{thm:rational}
Given $g(\zeta,0)$ of the form (\ref{structureg}) 
such that no two zeros of $g(\zeta,0)$ are related by $\omega_k=\omega_j^*$,
then for exactly one choice of $R(\zeta,t)$,
namely that given by (\ref{AB}), there exists a solution $g(\zeta,t)$ 
of  (\ref{PRGg}) which remains
on the original  rational form (\ref{structureg}). 

Necessary and sufficient condition for this rational solution to also solve the L\"owner-Kufarev equation (\ref{lk}) is that
$R(\zeta,t)=0$. This occurs precisely under the condition that
whenever $g(\zeta,t)$ has a zero $\omega_k$ in $\D$, the reflected point $\omega_k^*$
is a pole of $g(\zeta,t)$ of order strictly greater than that of the zero.
This property is conserved in time.

Every pole $\zeta_j$ of $g$ moves out from the origin, and every zero $\omega_k$ inside the unit disk moves towards
the origin, as time increases. 
\end{theorem}

The last statement follows from (\ref{omegazetaPPstar}) together with the fact that $P$ is positive in $\D$, negative outside, and the opposite for
$P^*$.

\begin{remark}
A particular consequence of Theorem \ref{thm:rational} is that the there are no polynomial solutions of the L\"owner-Kufarev equation  with
zeros of $g$ in $\D$. 
\end{remark}

\begin{example}\label{ex:onlyPG}
This example is supposed to illustrate Theorem~\ref{thm:rational}.
Let $g$ initially be given by
$$
g(\zeta,0)=b(0)\frac{\zeta-\omega_1(0)}{\zeta-\zeta_1(0)}
$$
for some $0<\omega_1(0)<1$, $\zeta_1(0)>1$, $b(0)>0$. Then, first of all, there exists a unique
solution of the Polubarinova-Galin equation of the same form
\begin{equation}\label{gmobius}
g(\zeta,t)=b(t)\frac{\zeta-\omega_1(t)}{\zeta-\zeta_1(t)}
\end{equation}
with $0<\omega_1(t)<1$, $\zeta_1(t)>1$, $b(t)>0$. The system of ordinary differential
equations in Theorem~\ref{thm:dynamics} for $\omega_1(t)$, $\zeta_1(t)$, $b(t)$
explicitly becomes
$$
\begin{cases}
\dot\omega_1=--\frac{q\zeta_1}{b^2}-2A_1+\frac{2A_1\omega_1}{\omega_1-\zeta_1},\\
\dot\zeta_1=-\frac{q\zeta_1}{b^2\omega_1}-\frac{A_1\zeta_1}{\omega_1}+\frac{2A_1\zeta_1}{\omega_1-\zeta_1},\\
\dot{b}=\frac{q\zeta_1}{b\omega_1}+\frac{A_1b}{\omega_1},
\end{cases}
$$
where
$$
A_1= \frac{q(\omega_1-\zeta_1)(1-\omega_1{\zeta}_1)}{b^2(1-|\omega_1|^2)},
$$
and we have taken into account that all quantities are real.
It is seen immediately that the solution will go on as long as $\omega_1(t)$, $\zeta_1(t)$, $b(t)$
stay in the above specified intervals.
However, the so obtained solution will not solve the L\"owner-Kufarev equation because, by Theorem \ref{thm:rational},
that equation requires that $g$ has a pole of order at least two at the reflected point of $\omega_1(t)$.

Still, the L\"owner-Kufarev equation will also have a solution, at least a weak solution. This solution will
represent an evolution on a Riemann surface $\mathcal{M}$ above $\C$ with a branch point over $f(\omega_1(t),t)$, which has to be a fixed
(time-independent) point. If $\zeta_1(0)\ne\omega_1^*(0)$ this solution (it will be unique after the Riemann surface
$\mathcal{M}$ has been fixed) will be of the form
$$
g(\zeta,t)= b(t)\frac{(\zeta-\omega_1(t))(\zeta-\omega_2(t))(\zeta-\omega_3(t))}{(\zeta-\zeta_1(t))(\zeta-\zeta_2(t))^2},
$$
where $\zeta_2(t)=\omega_1(t)^*$ and $\omega_2(0)=\omega_3(0)=\zeta_2(0)$.
If $\zeta_1(0)=\omega_1^*(0)$ it will be of the slightly simpler form
$$
g(\zeta,t)= b(t)\frac{(\zeta-\omega_1(t))(\zeta-\omega_2(t))}{(\zeta-\zeta_1(t))^2},
$$
with $\zeta_1(t)=\omega_1(t)^*$ and $\omega_2(0)=\zeta_1(0)$.
One then obtains the evolution in Example~\ref{ex:sakai}, 
where $\zeta_1$ was used as time parameter. Thus, by (\ref{exg}), (\ref{bat}), 
$\omega_2(t)=2t\zeta_1(t)-\zeta_1(t)^{-1}$, 
$b(t)=b(0)\zeta_1(0)^{-3}\zeta_1(t)^3$.

\end{example}


\subsection{Approach via quadrature identities}

This approach to rational solutions has the advantage that it can incorporate transitions of zeros through $\partial\D$, even when $q\ne 0$.

In the previous subsection we saw that structural properties such as
having an identity of the kind (\ref{sigma}), or an estimate
(\ref{crho}), holding are preserved in time when $f=f(\cdot,t)$
represents a weak solution. The same type of argument also shows that the
property of $g$ being a rational function is preserved, because
such property is equivalent to an identity (\ref{sigma}) holding
with $\sigma$ of a particularly simple form. We start by elaborating
a lemma making this statement precise.

When $g$ is rational the computation in the beginning of the proof
of Lemma~\ref{lem:radius} can be made more explicit and ends up with
a quadrature formula for $h\in\mathcal{O}(\overline{\D})$. Specifically we get
\begin{equation}\label{weightedqi}
\frac{1}{\pi}\int_\D h |g|^2 dm =\frac{1}{2\pi\I}\int_{\partial\D}
h{f^*}df = \sum\res_\D(h f^* g d\zeta)+\sum_j c_j\int_{\gamma_j} h g
d\zeta.
\end{equation}
Here the $\gamma_j$ are arcs in $\D$ connecting the points where
$f^*$ has logarithmic poles. The above computation actually does not
require $h$ to be holomorphic in $\D$, it is enough that $hg$ is holomorphic.
Thus one can allow $h$ to have a pole at any zero of $g$ in $\D$.

Equation (\ref{weightedqi}) is a Riemann surface version, pulled back to $\D$, of (\ref{qi}). 
In the terminology of \cite{Sakai-1988} the corresponding domain $\tilde{\Omega}=\tilde{f}(\D)$
($\tilde{f}$ being the lift of $f$, as in Section~\ref{sec:nonlocuniv})),
then is a {\bf quadrature Riemann surface}.
The formula (\ref{weightedqi}) may alternatively be
presented without line integrals by expressing the right member in
terms of an integral of $h$, namely
$$
H(z)=\int_0^z h(\zeta) g(\zeta) d\zeta.
$$
The quadrature identity then becomes
$$
\frac{1}{\pi}\int_\D h |g|^2 dm=-\frac{1}{2\pi\I}\int_\D dH\wedge
d\bar{f} =-\sum\res_\D(H
df^*)=\sum\res_\D\frac{H(\zeta)g^*(\zeta)d\zeta}{\zeta^2}.
$$

By spelling out the results of the residue calculations, and taking
into account the other direction we have the following.

\begin{proposition}\label{lem:qi}
Let $f\in\mathcal{O}_{\rm norm}(\overline{\mathbb{D}})$.
Then $g=f'$ is a rational
function if and only if there exist $\alpha_j$, $\gamma_j$, $a_{kj}$, $c_j$, $r$, $\ell$, $n_j$ so that the quadrature identity
\begin{equation}\label{weightedqi1}
\frac{1}{\pi}\int_\D h |g|^2 dm=\sum_{j=1}^r
c_j\int_{\gamma_j} h g d\zeta+
\sum_{j=0}^\ell\sum_{k=1}^{n_j-1} a_{jk} h^{(k-1)}(\alpha_j)
\end{equation}
holds for all $h\in\mathcal{O}(\overline{\mathbb{D}})$. Here we have used the
same numbering as in (\ref{qi}). In particular, $\alpha_0=0$.
The end points of the $\gamma_j$ are the logarithmic poles of $f^*$ and the $\alpha_j$
are the ordinary poles of $f^*g$ in $\D$. In other words, with $g$ of the form (\ref{structureg}),
these points are from the set $\{\zeta_0^*, \zeta_1^*,\dots,\zeta_\ell^*\}$, $\zeta_0=\infty$.

In addition, if $I\ni t \mapsto f(\cdot, t)\in \mathcal{O}_{\rm norm}(\overline{\mathbb{D}})$ 
represents a weak solution of the Hele-Shaw problem as in Proposition~\ref{prop:weaksolutionLK},
then the form (\ref{weightedqi1}) is stable over time, assuming that it holds initially.
However, the coefficients depend on time:
$a_{jk}=a_{jk}(t)$, $\alpha_j=\alpha_j(t)$, $\gamma_{j}=\gamma_{j}(t)$, with the qualification that  $\alpha_0=0$ is fixed
and that for $a_{01}$ we have the precise behavior  $a_{01}(t)=a_{01}(0)+2 Q(t)$.
(Thus $a_{01}(t)$ may become zero at one moment of time.)

\end{proposition}

\begin{proof}
For the first statement in the proposition, 
the `only if' part follows by evaluating the residues in the
previous formulas, the $\zeta_j$ being the poles of $f^*g$ in $\D$.
Note that a zero $\omega$ of $g$ in $\D$ will allow $f^*$ to have a pole at the
same point, and of the same order, hence $g$ to have a pole of one order higher
at the reflected point $\omega^*$, without causing a contribution in the right member of
(\ref{weightedqi1}). Alternatively, one may allow the test function $h$ to have a pole at
$\omega$.

To prove the `if' part we use in
(\ref{weightedqi1}) the test functions
$$
h(\zeta)=\frac{1}{z-\zeta} \quad (\zeta\in\D),
$$
with $z\notin \overline{\D}$. Then the left member of
(\ref{weightedqi1}) becomes the previously used (see
(\ref{cauchyG})) Cauchy transform $G$ of $|g|^2\chi_\D$ while the
right hand side takes the form $R(z)+Q(z)$, where $R(z)$ is a
rational function and $Q(z)$ is the contribution from the line
integrals:
$$
Q(z)=\sum_j c_j\int_{\gamma_j}\frac{g(\zeta)d\zeta}{z-\zeta}.
$$

Reasoning as in the proof of Lemma~\ref{lem:radius} we first get
$G=\bar{f}g+H$ on $\overline\D$ for some
$H\in\mathcal{O}({\D})$ which is continuous up to $\partial\D$, and
then the identity
$$
\bar{f}g+H =R+Q
$$
on $\partial\D$. The latter relation can also be written as
\begin{equation}\label{boundaryrelation}
{f^*(z)}=\frac{R(z)}{g(z)}-\frac{H(z)}{g(z)}+\frac{Q(z)}{g(z)} \quad
(z\in\partial \D).
\end{equation}
The integrals appearing in the definition of $Q(z)$ make jumps of magnitude $\pm
2\pi\I g(z)$ as $z$ crosses $\gamma_j$ from one side to the other.
It follows that the first two terms in the right member of
(\ref{boundaryrelation}) are meromorphic functions in $\D$ while the
last term is holomorphic except for constant ($=2\pi\I c_j$) jumps
across the arcs $\gamma_j$. These jumps disappear when
differentiating (\ref{boundaryrelation}). The conclusion is that
$df(z)=g(z)dz$ is a rational (Abelian) differential (or $f$ an Abelian
integral), because the right member gives the appropriate extension
of $f$ to the Riemann sphere.
Thus $g$ is a rational function, as claimed. Note that (\ref{boundaryrelation})
then holds  identically in $\C$.

The second statement in  the proposition, about weak solutions, is an easy consequence of (\ref{weakunitdisk2}).

\end{proof}

\begin{example}
With
$$
g(\zeta)=b\,\frac{(\zeta-\omega_1)(\zeta-\omega_2)}{(\zeta-\zeta_1)^2}
$$
the quadrature identity is in general of the form
$$
\frac{1}{\pi}\int_\D h|g|^2dm= a_0 h(0)+a_1 h(\zeta_1^*)+ c \int_0^{\zeta_1^*}hgd\zeta.
$$
However, if $\zeta_1^* =\omega_1$ (or $\zeta_1^*=\omega_2$) then $a_1=0$ and if $\zeta_1=\frac{1}{2}(\omega_1 +\omega_2)$
then $c=0$. Both of this occurred in Example~\ref{ex:sakai}.
Similarly, $a_1=c=0$ in case $\omega_1=\omega_2=\zeta_1$.

Taking the full Hele-Shaw evolution, as in Examples \ref{ex:R} and \ref{ex:sakai}, into account we 
therefore see that one can achieve a quadrature identity description of the evolution on the unified form
$$
\frac{1}{\pi}\int_\D h(\zeta)|g(\zeta,t)|^2dm(\zeta)= 2Q(t)h(0)
$$
for $0<t<\infty$, 
despite the fact that $f(\zeta,t)$ changes behavior as in (\ref{ftwocases}) when the zero of
$g$ passes through the unit circle.
\end{example}


\bibliography{bibliography_gbjorn}

\end{document}